\documentclass[]{amsart}
\usepackage{latexsym,fancyhdr,amssymb,color,amsmath,amsthm,graphicx,listings,comment}
\usepackage[section]{placeins}
\newtheorem{thm}{Theorem}
\newtheorem{lemma}{Lemma}
\newtheorem{propo}{Proposition}
\newtheorem{coro}{Corollary}
\newtheorem{conjecture}{Conjecture}
\let\paragraph\subsection

\title{Analytic torsion for graphs}
\author{Oliver Knill}
\date{January 23, 2022 }
\address{Department of Mathematics \\ Harvard University \\ Cambridge, MA, 02138 }

\begin{document}

\begin{abstract}
For any finite simple graph $(V,E)$, the squared analytic torsion is the
positive rational number $A(G) = \prod_k {\rm Det}(L_k)^{k (-1)^{k+1}}$,
where $L_k$ are the blocks of the Hodge Laplacian $L=D^2=(d+d^*)^2$ of the Whitney
complex and ${\rm Det}$ is the pseudo determinant. Torsion $A(G)$ agrees with the super pseudo determinant 
${\rm SDet}(D) = \prod_k {\rm Det}(D_k)^{(-1)^k}$ of the Dirac blocks
$D_k=d_k^* d_k$ of the Dirac operator $D=d+d^*$ and is related to the pseudo determinant 
${\rm Det}(D)=\pm \prod_k{\rm Det}(D_k)$ of $D$. 
This gives a generalized matrix tree theorem: $A(G)$ is the ratio of rooted spanning trees 
on even-dimensional simplices divided by the number of rooted spanning trees in odd simplices. 
In particular, the classical matrix tree theorem rephrases that for graphs without triangles, 
$A(G)$ is the number of rooted spanning trees in $G$. For $2$-spheres with $|F|$ triangles, 
torsion is $A(G)=|V|/|F|$, rephrasing Von Staudt's theorem that the number of 
spanning trees in a $2$-sphere $G$ 
and its dual graph $G'$ agree. We prove in general $A(G)=|V|$ for graphs 
homotopic to $1$ and $A(G)=|V|/|V'|$ for $(2r)$-spheres and $A(G)=|V| |V'|$ for 
$(2r+1)$-spheres, where $V'$ is the set of maximal simplices in $G$ and $|V|,|V'|$ are
the cardinalities $G$ or $G'$. 
Torsion, as the super pseudo determinant of the Dirac operator 
$D=d+d^*$ can be defined for any bounded differential complex. Similar formulas hold so for Wu torsion 
of spheres in the Wu complex. We also start to look at the expectation of $A$ on Erdoes-Renyi 
probability spaces or look into the problem which graphs on $n$ vertices maximize or minimize $A(G)$. 
The limit $\lim_{n \to \infty} A(G_n)$ for Barycentric refinements of 
even dimensional spheres can be computed.
\end{abstract}

\maketitle

\section{Summary}

\paragraph{}
The {\bf squared analytic torsion} for a graph $G=(V,E)$ is defined as 
the spectral quantity $A(G) = \prod_k {\rm Det}(L_k)^{k (-1)^{k+1}}$,
where the {\bf $k$-form Laplacians} $L_k=d_k^* d_k + d_{k-1} d_{k-1}^*$ is the Laplacian on $k$-forms
and where ${\rm Det}$ denotes the {\bf pseudo determinant}, the product of the non-zero eigenvalues.
The matrices $L_k$ are the {\bf Hodge blocks} in the {\bf Hodge Laplacian} $L=(d+d^*)=D^2$, where $D=d+d^*$ is the 
{\bf Dirac operator} of $G$. The {\bf exterior derivatives} $d_k$ from $k$-forms to $(k+1)$-forms define the 
{\bf Dirac blocks} $D_k=d_k^* d_k$ and $D_k'=d_k d_k^*$ making up $L_k = D_k + D_{k-1}'$. 
Now, $A(G) = {\rm SDet}(D)$, where ${\rm SDet}(D)=\prod_k {\rm Det}(D_k)^{(-1)^{k+1}}$ is the
{\bf super determinant of $D$}. Cauchy-Binet for pseudo determinants
\cite{CauchyBinetKnill} allows now to see ${\rm SDet}(D)$ as a super count of trees because
${\rm Det}(d_k^* d_k) = \sum_{P} {\rm det}(d_k(P))^2$ is a sum of squares of minors 
${\rm det}(d_k(P))^2 \in \{0,1\}$. Both in the contractible as well as the sphere case, we can
identify the even and odd part as complementary trees, where $|V|$ or $|V|'$ reduce from rooted trees
to trees.  This insight allows to count $A(G)$ for contractible graphs by shaving off the first
row and column of $D$, a process which divides the super count by $|V|$. For spheres, where shaving of the 
both first and last column and first and last row of $D$ divides the super count by $|V||V|'$ or $|V|/|V|'$ 
depending on dimension. The resulting quantity $\phi(G)$ is a ratios of trees and complementary trees,
and is $1$ both for contractible graphs or spheres.

\paragraph{}
To formulate the results, a few more definitions are needed. 
The {\bf unit sphere} $S(v)$ of a vertex $v \in V$ in $G$ is the graph induced by the neighbors 
of $v$. The {\bf dual graph} $G'$ of $G$ has the set $V'$ of maximal faces = {\bf facets} 
$x$ as vertices, where two facets $x,y$ are connected if $x \cap y$ has co-dimension $1$. 
A {\bf $k$-sphere} $G$ is inductively defined as a graph for which all unit spheres 
$S(v)$ are $(k-1)$-spheres and such that $G-v$ is contractible for some vertex $v$. 
A graph $G$ is inductively defined to be {\bf contractible} 
if there exists $v \in V$ such that $G-v$ and $S(v)$ are both contractible for some $v$. 
To start the definitions, the empty graph $0$ is the $(-1)$-sphere and
the $1$-point graph $1$ is contractible. A graph is {\bf homotopic to $1$} if a finite number of 
contractions $G \to G-x$ or extensions $G \to G +_H x$ with contractible subgraph $H$, 
lead to $1$. We write $|V|$ and $|V'|$ for the vertex cardinalities of $G$ and $G'$. These are
the first and last entries in the $f$-vector $f=(f_0,f_1,\dots,f_m)$ of $G$, where 
$f_k$ countis the number of complete sub-graphs $K_{k+1}$ in the graph $G$. 
It should be clear that we can replace the clique complex of $G$ with an arbitrary 
simplicial complex. This generalization is almost equivalent as the Barycentric refinement of any 
finite abstract simplicial complex is a Whitney complex. 

\begin{center}\fbox{\parbox{10cm}{
{\bf Theorem:} If $G$ is homotopic to $1$, then $A(G)=|V|$. 
For triangle-free graphs, $A(G)$ is the number of rooted spanning trees in $G$. 
For $(2r+1)$ spheres, $A(G)=|V| |V'|$. For $(2r)$-spheres, $A(G)=|V|/|V'|$. 
For $2$-spheres, $A(G)$ is the ratio of the number of spanning trees in $G$ and the number
of spanning trees in the dual graph $G'$.
}}
\end{center}

\paragraph{}
The use of {\bf pseudo determinant} is crucial because the Hodge blocks $L_k$ as well as the Dirac blocks 
$D_k$ or $D_k'$ building up the Hodge blocks as $L_k=D_k + D_{k-1}'$ are in general singular.
In order to define $A(G)$ and to identify $A(G) = {\rm SDet}(D)$ as a super determinant,
no assumptions whatsoever on $G$ and especially no assumption on the Betti numbers 
$b_k = {\rm dim}({\rm ker}(L_k))$ are necessary. 
By {\bf McKean-Singer symmetry}, also in full generality, the super determinant of the Hodge operator
${\rm SDet}(L) = \prod_k {\rm Det}(L_k)^{(-1)^{k+1}}=1$ for all finite simple graphs. McKean-Singer 
is like Poincar\'e duality an involution symmetry but 
it holds for arbitrary graphs; in comparison, almost all graphs lack Poincar\'e duality.
If $F_k=d_k+d_{k-1}^*$ is the $(n \times f_k)$-matrix
consisting of the {\bf Dirac columns} belonging to $k$-dimensional simplices, then the Hodge
blocks $L_k=F_k^* F_k$ satisfy by {\bf Cauchy-Binet} identity for 
pseudo determinant ${\rm Det}(L_k) ={\rm Det}(F_k^* F_k) 
= \sum_{|P|=f_k-b_k} {\rm det}(F_k(P))^2$ \cite{CauchyBinetKnill}. The classical
Cauchy-Binet theorem which involves classical determinants would not apply as it would
require one of the matrices $F_k^* F_k$ or $F_k F_k^*$ to have full rank. 
We need here the {\bf pseudo determinant Cauchy-Binet}
\cite{CauchyBinetKnill}. We are still stuck at this point because 
the minors ${\rm det}(F_k(P))$ can be pretty arbitrary, so that ${\rm Det}(L_k)$ does 
not count things yet. 

\paragraph{}
Fortunately, the mathematics becomes better with the {\bf Dirac blocks} $D_k = d_k^* d_k$ and 
which are paired with their isospectral block $D_k' = d_k d_k^*$. 
While $D_k$ is a $(f_k \times f_k)$-matrix, the $D_k'$ is a 
$(f_{k+1} \times f_{k+1})$-matrix. We have now the important identity
${\rm Det}(L_k)={\rm Det}(d_k^* d_k) {\rm Det}(d_{k-1} d_{k-1}^*)$ because
$L_k = d_k^* d_k + d_{k-1} d_{k-1}^*$ can be written as $L_k=F_k^* F_k$. The Hodge block $L_k$ is 
essentially isospectral to $L_k'=F_k^* F_k$ (meaning that it has the same non-zero eigenvalues).
The matrix $L_k'$ is a $n \times n$ matrix with two Dirac blocks 
$D_k,D_{k-1}'$ in the diagonal and everything else is zero. 
The {\bf Hodge block determinants} are now a product of {\bf Dirac block determinants}
${\rm Det}(L_k)={\rm Det}(D_k) {\rm Det}(D_{k-1})$. This important identity implies also 
the McKean-Singer symmetry ${\rm SDet}(L)=1$ again and also allows to see 
$A(G)= {\rm SDet}(D)$ is the super determinant of the Dirac operator $D=d+d^*$ of chain complex:

\begin{center}\fbox{\parbox{10cm}{
{\bf Key Lemma:} $A(G)={\rm SDet}(D) = \prod_{k=0}^r {\rm Det}(D_k)^{(-1)^{k+1}}$ 
}} \end{center}

\paragraph{}
The incidence matrices $d_k$ are not square matrices so that their determinant is not
defined a priori. But it is custom to define it as the square root of the determinant of 
$d_k^* d_k$ which is a square root. 
One can still define ${\rm Det}(d_k) = \sqrt{\rm Det}(D_k)$ in order to interpret
$\sqrt{A(G)}={\rm Sdet}(d)$. The square root of $A$ is known as a
{\bf determinant of the chain complex}. Classically, analytic torsion is defined as
$\sqrt{A(G)}$, but as pointed out before, we prefer in combinatorics to have
rational numbers. As a pseudo super determinant of a Dirac matrix it is an object
of classical linear algebra. So, we look at the super determinant of the Dirac operator 
rather than the super determinant of the exterior derivative.
It is no big deal.  Considering the squared analytic torsion
is a bit like looking at variance rather than standard deviation or looking at energy
functionals rather than length functionals. 

\paragraph{}
Again by the {\bf Cauchy-Binet theorem for pseudo determinants}, we have 
${\rm Det}(D_k) ={\rm Det}(d_k^* d_k) = \sum_{|P|=Rank(d_k)} {\rm det}(d_k(P))^2$. But 
now, ${\rm det}(d_k(P))= \pm 1$ allows to see the {\bf squared Dirac minors} ${\rm Det}(D_k)^2$ 
as counting the {\bf number of trees} in the simplex graph with $k$-simplices as vertices and 
$(k+1)$-simplices as edges. Actually $\prod_{k \; {\rm odd}} {\rm Det}(D_k)$ is the number
of rooted spanning trees in the graph $\mathcal{F}$ in which the odd-dimensional simplices are the
vertices and where two are connected if their intersection is a co-dimension $1$ simplex.
Similarly, $\prod_{k \; {\rm even}} {\rm Det}(D_k)$ is the number of rooted spanning trees
in the graph $\mathcal{B}$, in which the even-dimensional simplices are the vertices. Now,
if we include as ``rooted" in the sphere case that both a vertex as well as a facet is fixed.

\begin{center}\fbox{\parbox{10cm}{
{\bf Super matrix tree theorem}: For all graphs, 
$A(G)$ is the number of rooted spanning trees in $\mathcal{B}$ divided by the number of rooted 
spanning trees in $\mathcal{F}$. 
}} \end{center}

\paragraph{}
For {\bf triangle-free graphs}, torsion $A(G)={\rm Det}(L_0)$ is the number of {\bf rooted
spanning trees} in the graph $G$ itself. The integer ${\rm Det}(L_0)/|V|$ is
the number of {\bf (non-rooted) trees} by the {\bf matrix tree theorem} and each tree belongs to $|V|$
{\bf rooted trees}. A simple example is $A(C_n)=n^2$. An other example is provided by
a bipartite graph $K_{n,m}$ which is triangle-free and for which we have
$A(K_{n,m}) = n^{m-1} m^{n-1} (n+m)$. For the {\bf utility graph} $K_{3,3}$ for example, 
$A(K_{3,3})=3^4*6=486$. This is an interesting case because it is
the {\bf maximum} of the torsion functional $A$ on all graphs with $|V|=6$ vertices. 
We also know that for complete graphs, $A(K_n)=|V(K_n)|=n$. The octahedron graph
$A(G)=3/4$ has {\bf minimal torsion} among all graphs with $|V|=6$ vertices. 
While minima of torsion seem in general to gravitate to graphs with the homotopy type of bouquets 
of spheres, the bipartite graphs $K_{n,m}$ are good candidates for maxima on graphs with $|n+m|$. 
In any case, we have graphs $G=K_{n,n}$ of order $2n$ with $A(G) =2 n^{2n-1}$ showing that 
torsion can grow {\bf super exponentially} with the order $n$. Bouquets of 2-spheres are examples
where $A(G_n)$ decays exponentially with the order $n$.
%  Do[Print[N[ AnalyticTorsion[BouquetGraph3D[n]] ]],{n,10}]

\paragraph{}
For {\bf $2$-spheres}, graphs for which every unit sphere is a circular graph 
with $4$ or more vertices, we can quickly prove $A(G)=|V|/|V'| = f_0(G)/f_2(G)$. 
It rephrases the fact that the number of spanning trees in a $2$-sphere $G$
is the same than the number of spanning trees in the dual graph $G'$. 
We know that for all graphs homotopic to $1$, we have 
$A(G)=|V|$ and that if $|V'|$ is the number of maximal simplices in $G$,
then $A(G)=(|V| |V'|)$ for odd-dimensional spheres and
$A(G)=|V|/|V'|$ for even-dimensional spheres. 
%  A=BarycentricOperator[6]; f=Table[1,{7}]; Do[ Print[N[f[[1]]/f[[7]],20]]; f=A.f,{100}]
More generally, we can reformulate our main theorem as a statement which is easier to prove
and which generalizes the Von Staudt theorem for 2-dimensional spheres. 

\begin{center}\fbox{\parbox{10cm}{
{\bf Duality theorem}: For contractible graphs and for spheres, the
number of spanning trees in $\mathcal{B}$ is equal to the number of spanning 
trees in $\mathcal{F}$. 
}} \end{center}

\begin{figure}[!htpb]
\scalebox{0.6}{\includegraphics{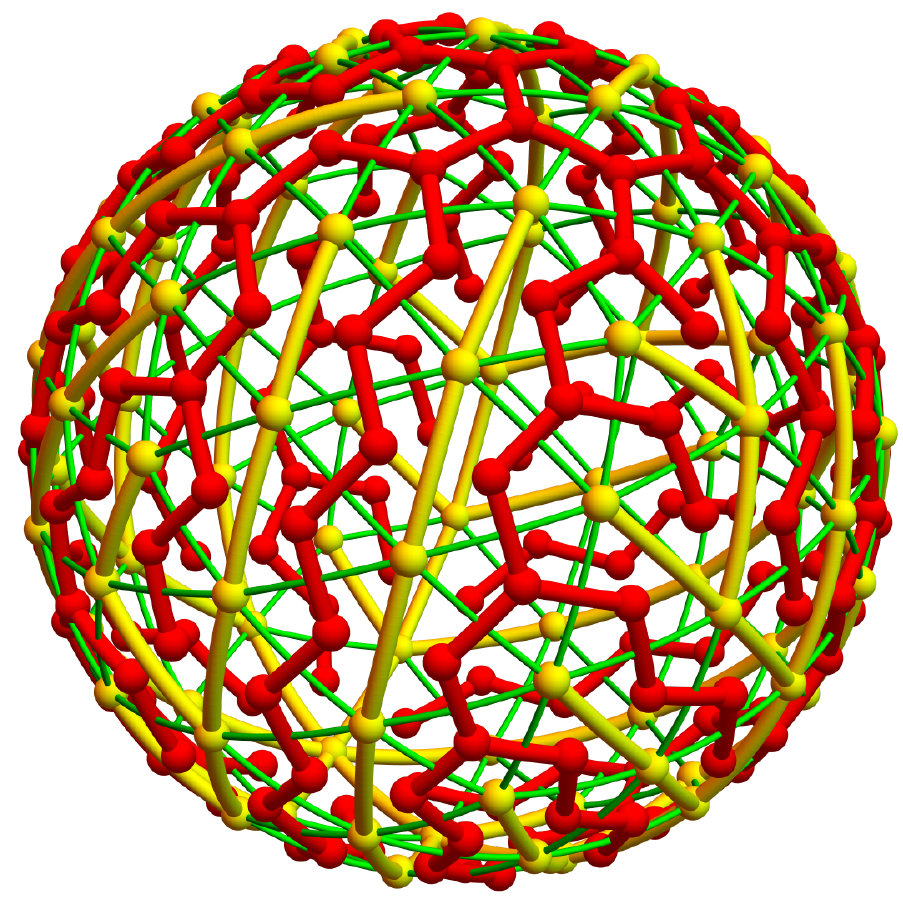}}
\label{Figure 0}
\caption{
A 2-sphere $G$ with a spanning tree $T$ defines
a dual spanning tree $T'$. 
}
\end{figure}

\paragraph{}
The duality result is false for  general graphs. 
For the house graph $G$ for example, we have $3$ spanning trees in $G'$ and 
$5$ spanning trees in $G$. The relation to the super matrix tree statement is that in
the contractible case, ``rooted" means
fixing a  root in the vertex set. In the sphere case, fixing a root in the vertex set
and fixing a root in the dual vertex set has different effects depending on dimension.
In the odd-dimensional case, both roots change the spanning trees in $\mathbb{B}$ and
do  not affect the spanning trees in $\mathbb{F}$. In the even dimensional case, 
fixing the root in $V=V(G)$ affects the spanning trees in $\mathbb{B}$ while fixing
the root in $V'=V(G')$ affects the spanning trees in $\mathbb{F}$. In the two-dimensional
case, the spanning trees in $\mathbb{B}$ are the spanning trees in $G$ while the 
spanning trees in $\mathbb{F}$ are the spanning trees in $G'$. 
What happens in a case like a 2-torus surface is that we would have to fix more roots
and that the number of possibilities to snap edges to break render the graph contractible
matters. 

\paragraph{}
The torsion functional $A(G)$
on $2$-dimensional surfaces different from spheres becomes complicated in general
and depends on the structure of the graph not only on the topology.  
For flat discrete $2$-tori, we measure $A(G)=|V|/(6 |V'|)$ but this
changes under deformations and Barycentric refinements already.
The functional $A(G)$ is interesting for discrete 2-manifolds already graphs for 
which all unit spheres are circular graphs with 4 or more elements. For a $2r$ sphere $G=G_0$
we can compute the limit $A(G_n)$ of Barycentric refinements: if $A$ is the linear Barycentric
refinement operator satisfying $A f(G_n)=f(G_{n+1})$, then $A$ is upper triangular with 
eigenvalues $\{ k!\, k=1,...,2r \}$. If $g=(g_0,...,g_{2r})$ is the 
eigenvector of the largest eigenvalue $r!$, then $\lim_{n \to \infty} A(G_n) \to g_{2r}/g_0$.
For odd dimensional spheres $G=G_0$, we  have $\lim_{n \to \infty} A(G_n) \to \infty$. 
Historically, torsion was never intended for even-dimensional manifolds as it is zero. 
In our case, $A(G_n)$ converges even in the Barycentric limit for even dimensional spheres
and $\sqrt{A(G_n)}/{\rm Vol}(G_n)$ in the Barycentric limit. 

\paragraph{}
Torsion $A(G)$ is neither a {\bf homotopy invariant} nor a {\bf valuation}. It does not
satisfy any sort of {\bf Meyer-Vietoris} relation in the combinatorial version we look at. 
The proof of the theorem shows this. For illustration, start with the {\bf icosahedron} 
$G$ which is a $2$-sphere with $f$-vector $f_G=(12,30,20)$ and Betti vector $b_G=(1,0,1)$, the torsion 
is $A(G)=12/20=3/5$. For the {\bf icosahedron with hair} $G_0=G \cup_H v$, a cone extension over the 
graph $H$ generated by a single vertex $\{v_1\}$gives $f_{G_0}=(13,31,20)$ and $A(G_0)=13/20$. The super 
count still works. For an {\bf icosahedron with nose} $G_1 = G \cup_H v$ which is a cone extension 
$H=\{v_1,v_2, (v_1,v_2) \in E\}$ over an edge have $f_{G_1}= (13,32,21)$ and still $A(G_1)=13/20$, 
even-so there are $21$ faces.  

\paragraph{}
A drastic change happens for the {\bf icosahedron with hat} after a cone extension 
$G_2 = G \cup_H v$ over a face sub-graph $H=\{v_1,v_2,v_3\}$ (this can be seen as a refinement of 
a triangle but it increases the dimension of the complex). Now, $f_{G_2}=(13,33,23,1)$ and
$b_{G_2})=(1,0,1,0)$ as a third block $L_3$ has appeared, even so $L_3$ is invertible.
Torsion $A(G_2) = 52/79$ has lost its relation with $f_{G_2}$ or $f_G$ as some
Dirac block $D_k$ in the interior would need cutting, messing up the super count.
Non-trivial cohomology or even non-trivial homology groups like for the projective plane do the same. 
Let us look at $G_3$ which is an {\bf icosahedron with ear}, a 2-sphere with the addition 
of a $1$-dimensional handle. Now, $f_{G_3}=(14,33,20)$ and $b_{G_3}=(1,1,1)$, 
$\chi(G)=14-33+20=1-1+1=1$ and $A(G_3) = 707/300$. 
The super-count of trees is again messed up. We can however glue an arbitrary number of
trees on $G$ to get a graph $G_4$ and still have $A(G)=|V|/|F|$. 

\begin{figure}[!htpb]
\scalebox{0.4}{\includegraphics{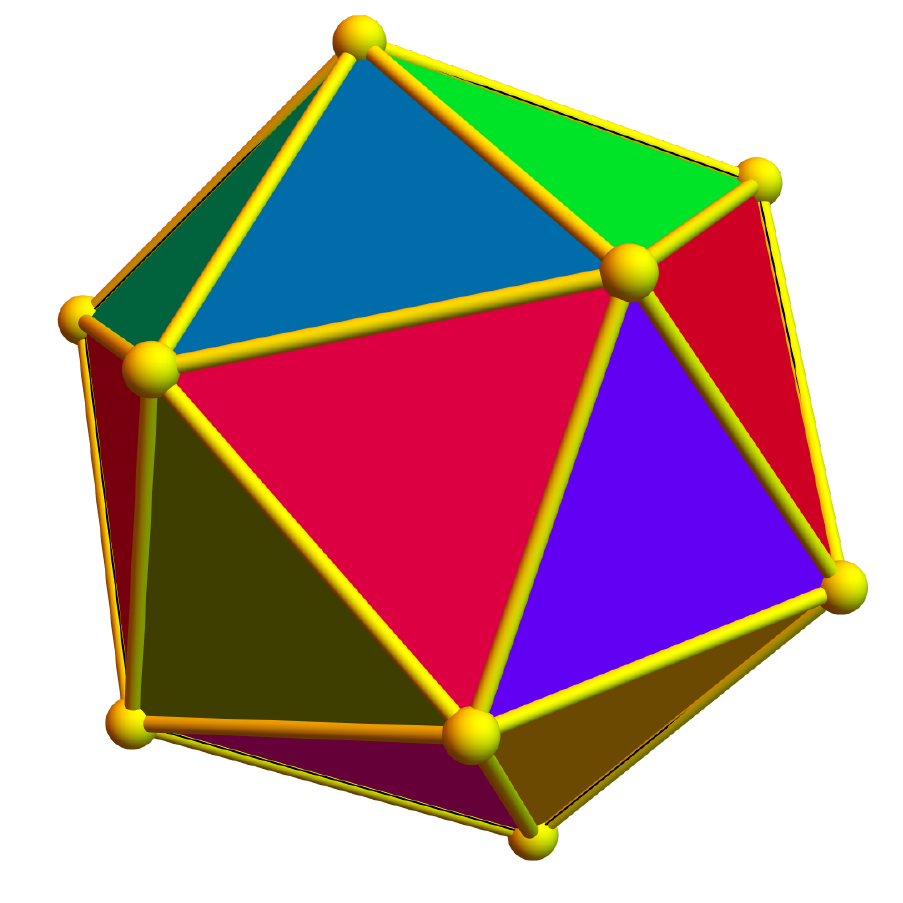}}
\scalebox{0.4}{\includegraphics{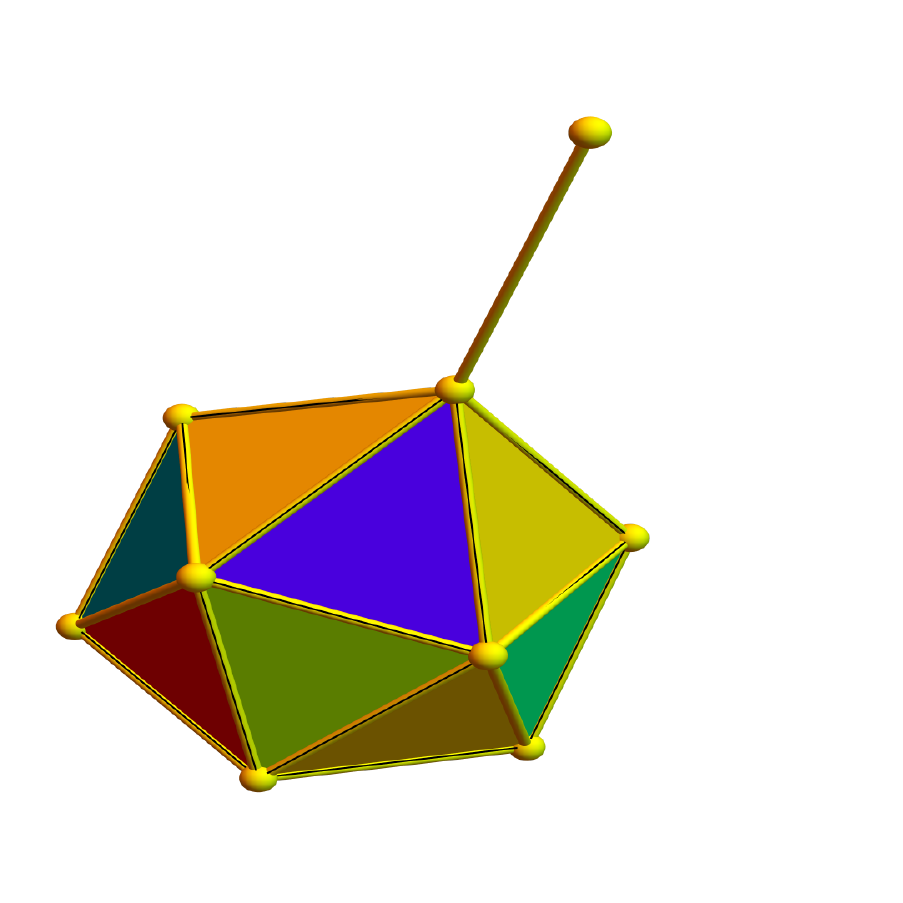}}
\scalebox{0.4}{\includegraphics{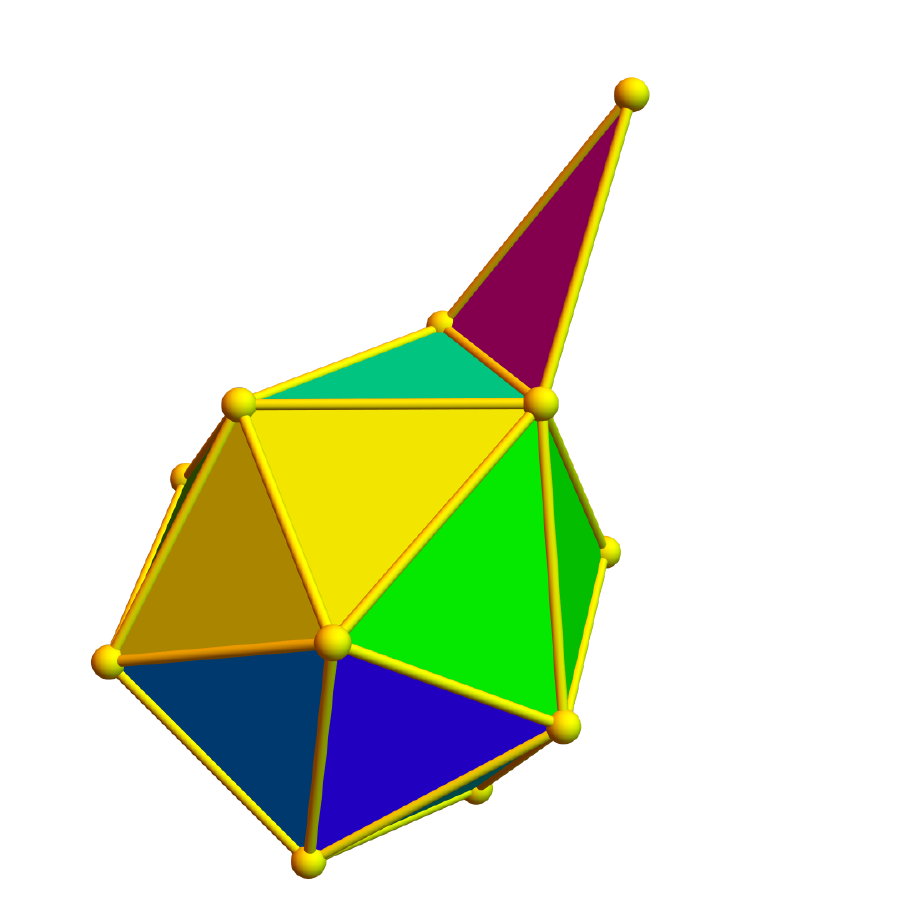}}
\scalebox{0.4}{\includegraphics{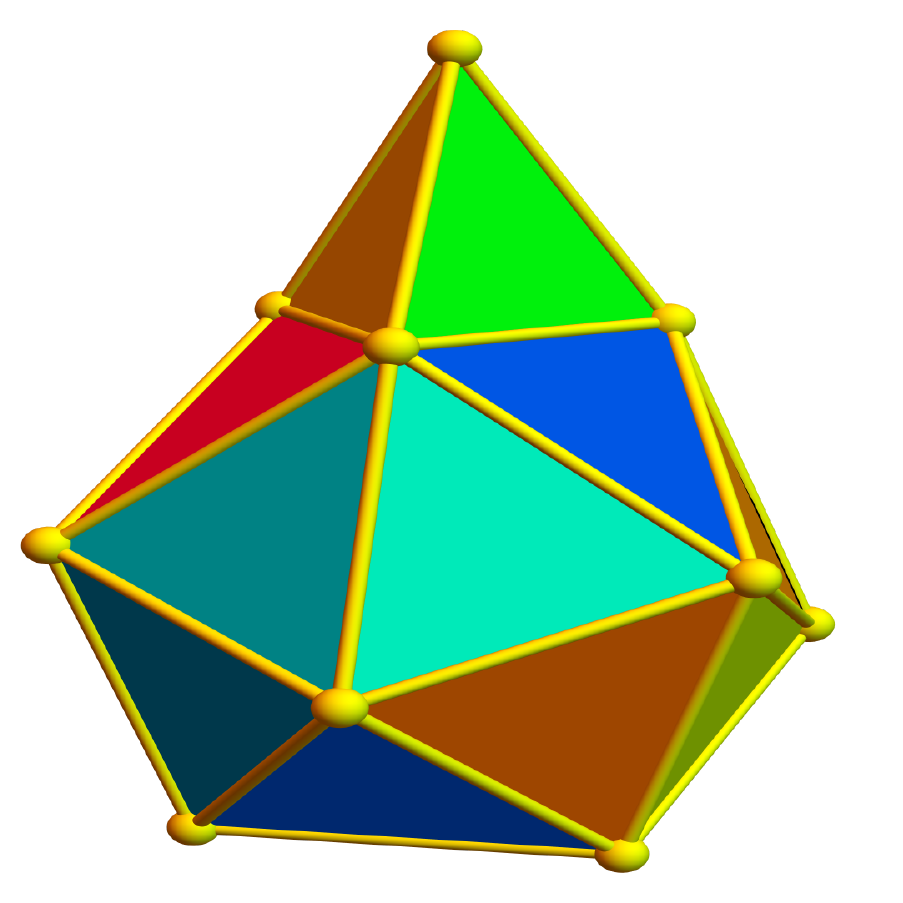}}
\scalebox{0.4}{\includegraphics{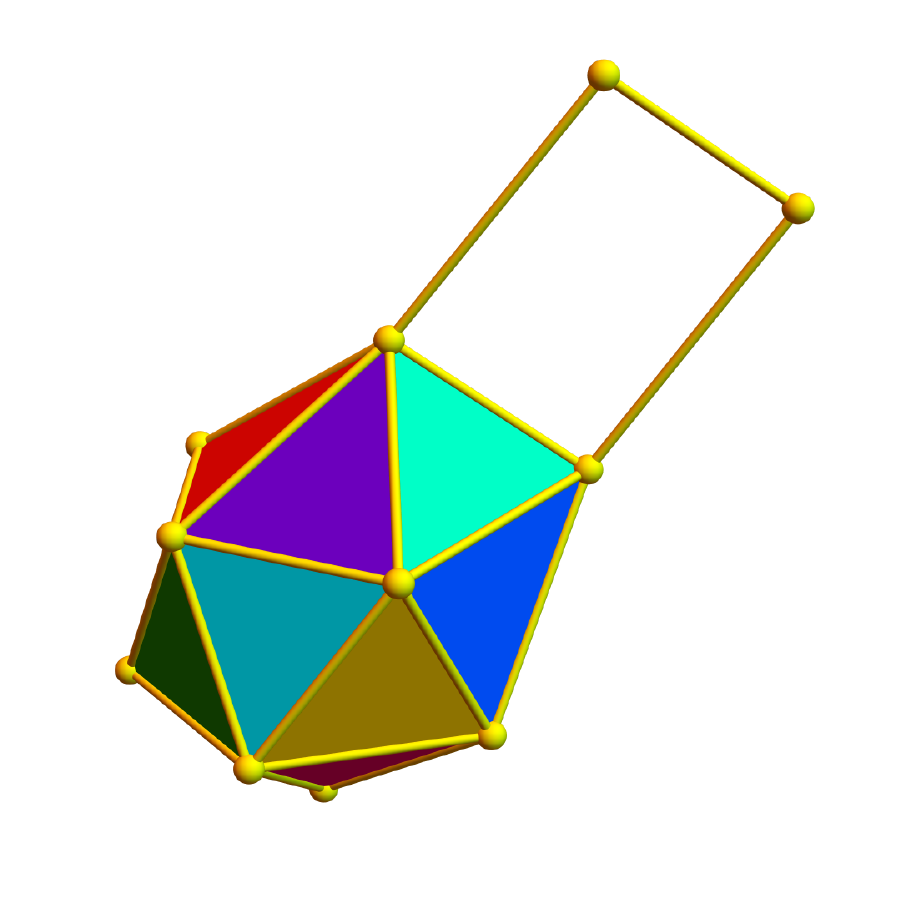}}
\scalebox{0.4}{\includegraphics{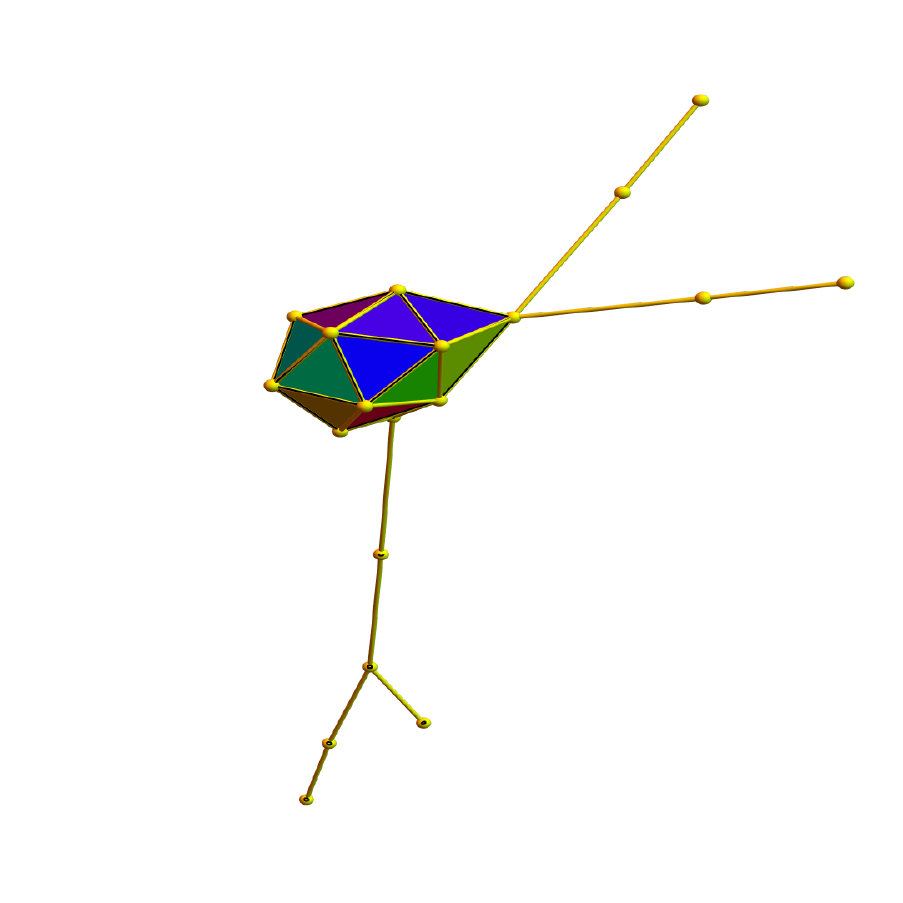}}
\label{Figure 0}
\caption{
We see modifications of an icosahedron $G$ with $f=(12,30,20)$ and $A(G)=|V|/|F|=12/20$. 
Adding a single hair $G_0$, multiple hairs $G_4$ or a nose $G_1$ gives 
$A(G_0)=A(G_1)=A(G_4) = |V(G_0)|/|F|=13/20$.
Torsion becomes complicated when adding a ``hat" $\hat{G_2}$ or an ``ear" $\hat{G_3}$: 
a modification involving faces of $G$ or changing the cohomology messes up 
torsion. Here, $A(G_2)= 52/79$ or $A(G_3) = 707/300$. 
}
\end{figure}

\begin{figure}[!htpb]
\scalebox{0.6}{\includegraphics{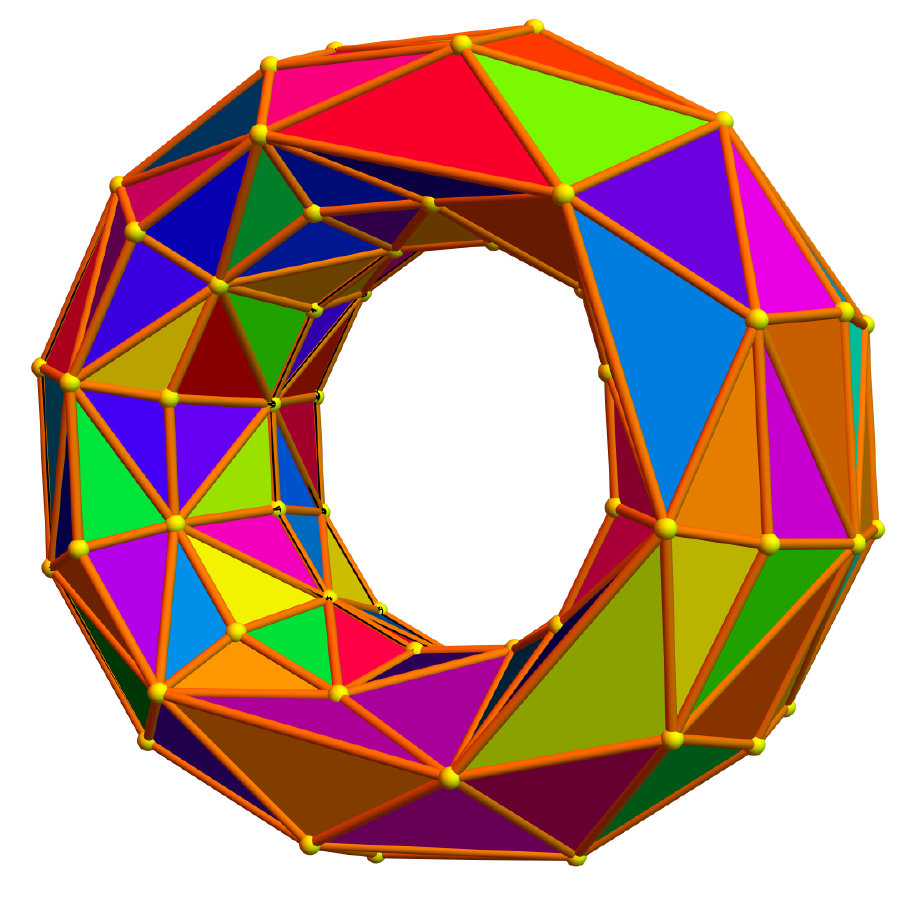}}
\label{Figure 0}
\caption{
Already for a 2-torus $G$, the functional $A(G)$ can take different values. 
For a flat Clifford torus, where each vertex degree is $6$, we have 
$A(G)=|V|/(3|F|)$. There is one with $f_G=(16,48,32)$, where the Baryentric refinement has
$f_{G_1}=(96,288,192)$ and $A(G)=3/32$. 
For the torus shown here and where the vertex degrees are either $4$ or $6$, we 
have the $f$-vector $f=(80,240,160)$ and Betti vector $b=(1,2,1)$ and 
analytic torsion $A(G)=18278388315/141574481716$.  
% s=Barycentric[TorusGraph[4,4]]; f=Fvector[s]; AnalyticTorsion[s]*(3*f[[3]]/f[[1]])
% s=Barycentric[TorusGraph[4,4]]; f=Fvector[s]; AnalyticTorsion[s]*(3*f[[3]]/f[[1]])
}
\end{figure}

\begin{figure}[!htpb]
\scalebox{0.6}{\includegraphics{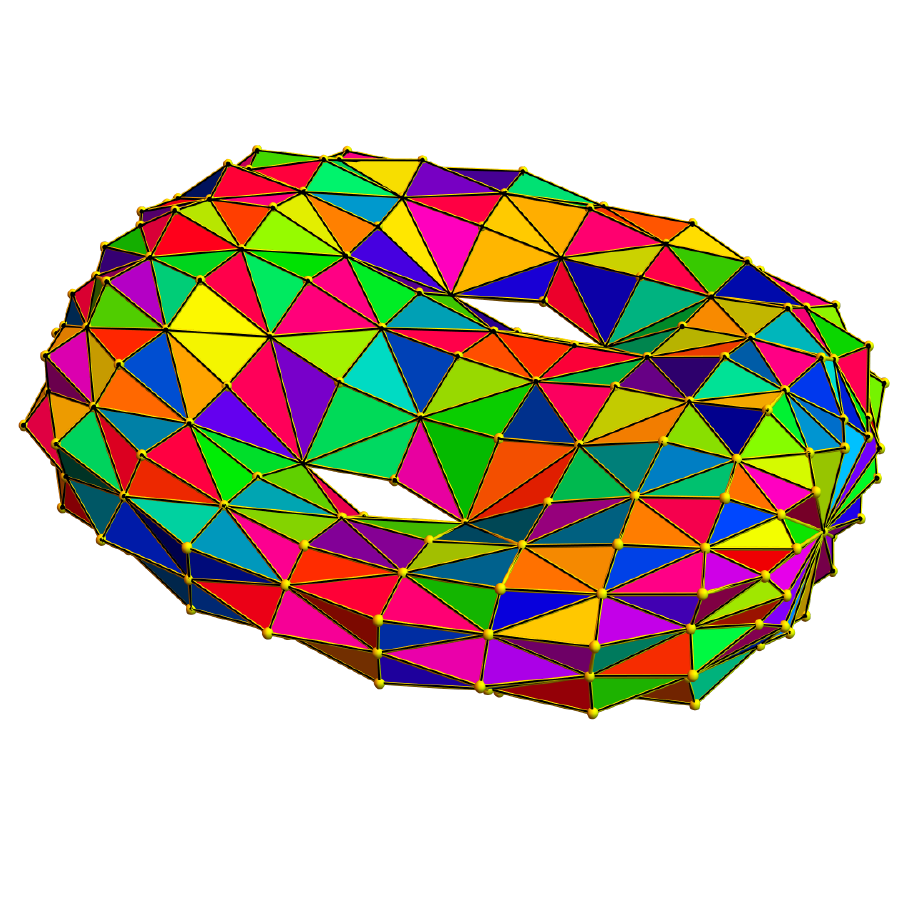}}
\label{Figure 0}
\caption{
A 2-manifold of genus $2$, realized as a graph $G$ with $f$-vector $f=(|V|,|E|,|F|)=(262,792,528)$
Betti vector $b=(1,4,1)$ and Euler characteristic $\chi(G) =|V|-|E|+|F|=1-4+1=-2$. A graph is a 2-manifold
as very unit sphere is a circular graph of length $4,6,8,16$ or $24$
leading to curvatures $\left\{-3,-\frac{5}{3},-\frac{1}{3},0,\frac{1}{3}\right\}$.
The Dirac operator $D$ is a $1582 \times 1582$ matrix. The analytic torsion is
\begin{tiny}
$\frac{70039080674189248816744297898336922150488256905}{4925851994736661747496162159567655905824021687407} \sim 0.0142187$.
\end{tiny}
The Dirac determinant is a huge $362$ digit integer $(1.30434...)*10^{362}$.
}
\end{figure}

\begin{figure}[!htpb]
\scalebox{0.6}{\includegraphics{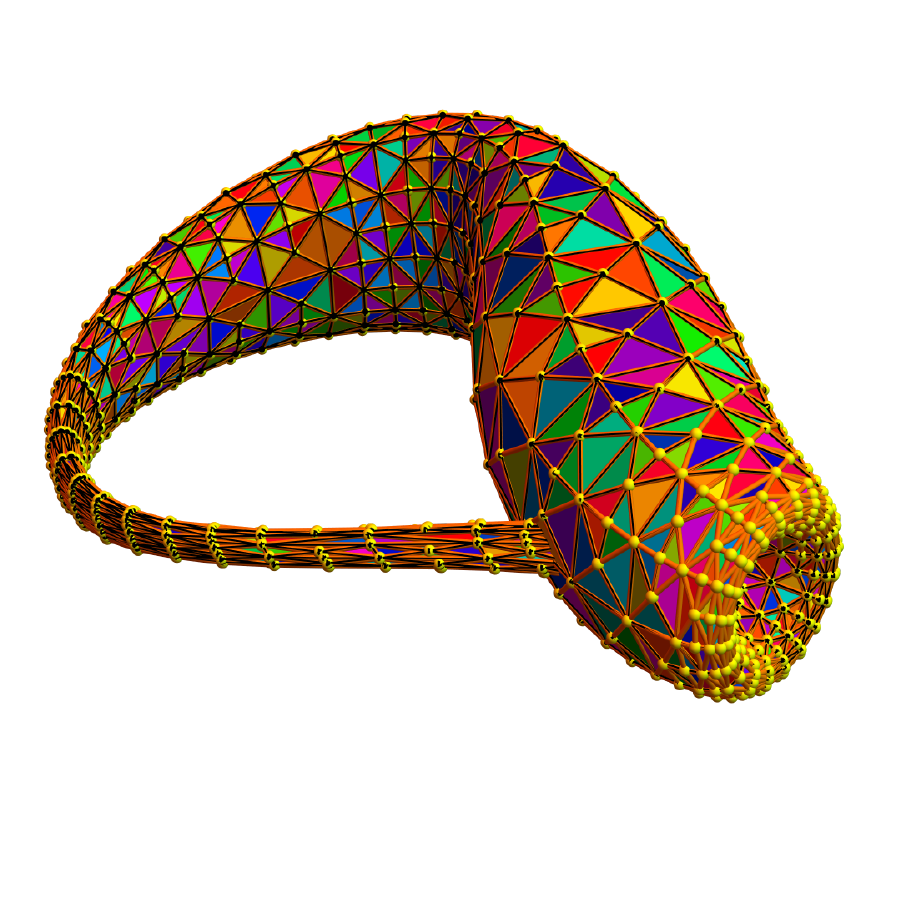}}
\label{Figure 0}
\caption{
We see a Klein bottle $G$. For a smaller version with $f$-vector $f=(|V|,|E|,|F|)=(50,150,100)$ 
we can compute everything like Betti vector $b=(1,1,0)$ and Euler characteristic 
$\chi(G) =|V|-|E|+|F|=1-1=0$. Its Dirac operator $D$ is a $300 \times 300$ matrix.
We measure $A(G) = |V|/8$. But similarly as for a flat Clifford torus, we have
here a rather uniform situation, where half the vertices have vertex degree $8$ 
and half have vertex degree $4$. 
}
\end{figure}

\begin{figure}[!htpb]
\scalebox{0.6}{\includegraphics{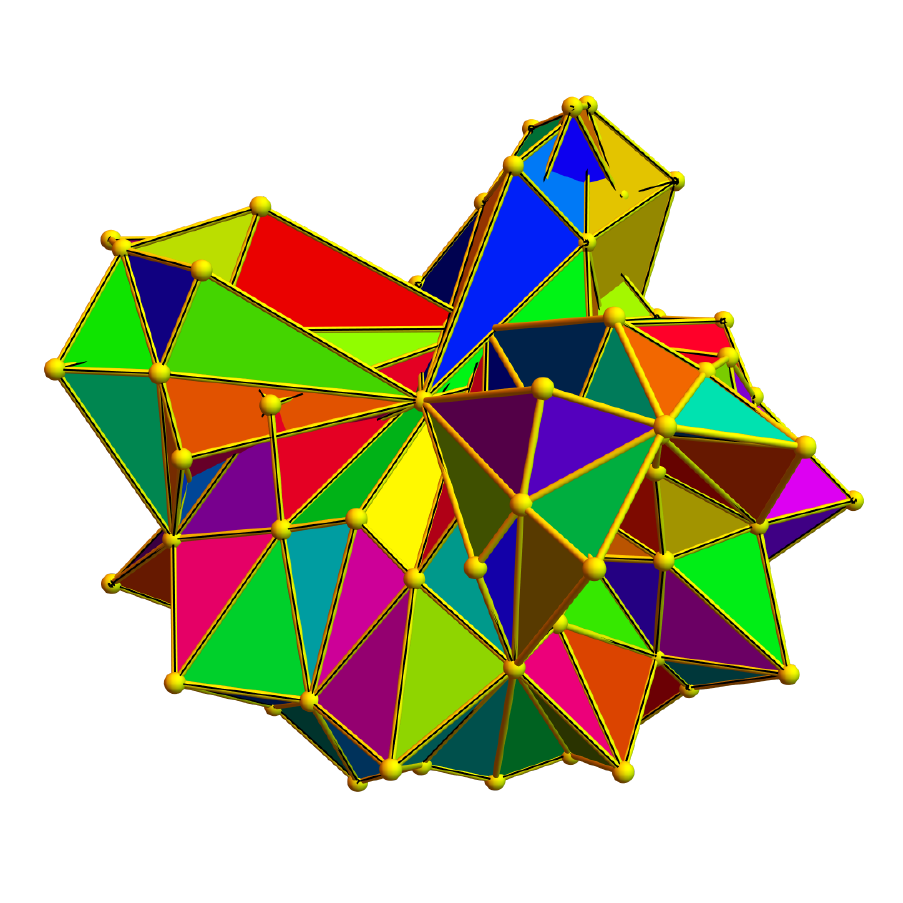}}
\label{Figure 0}
\caption{
A dunce $G$ hat with $f$-vector $f=(105, 320, 216)$, Betti
vector $(1,0,0)$ and torsion $A(G)=105$. The graph is an example of a graph
which is not contractible but which is homotopic to $1$. The theorem still
applies. 
}
\end{figure}

\begin{figure}[!htpb]
\scalebox{0.5}{\includegraphics{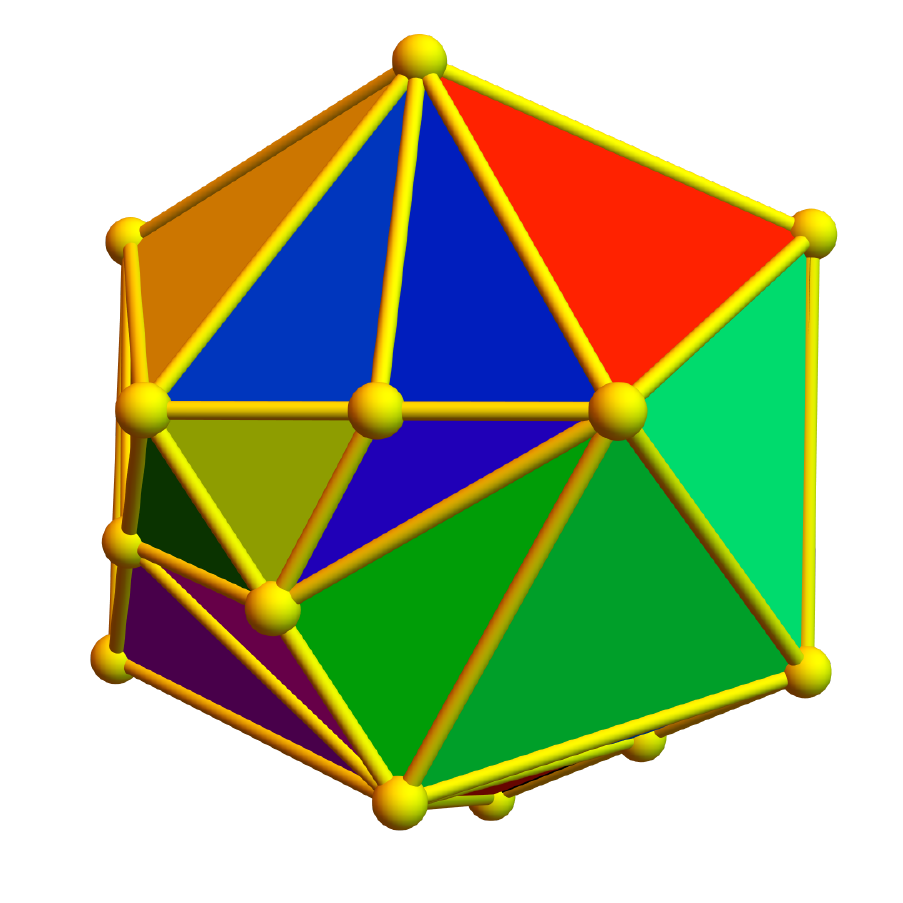}}
\scalebox{0.5}{\includegraphics{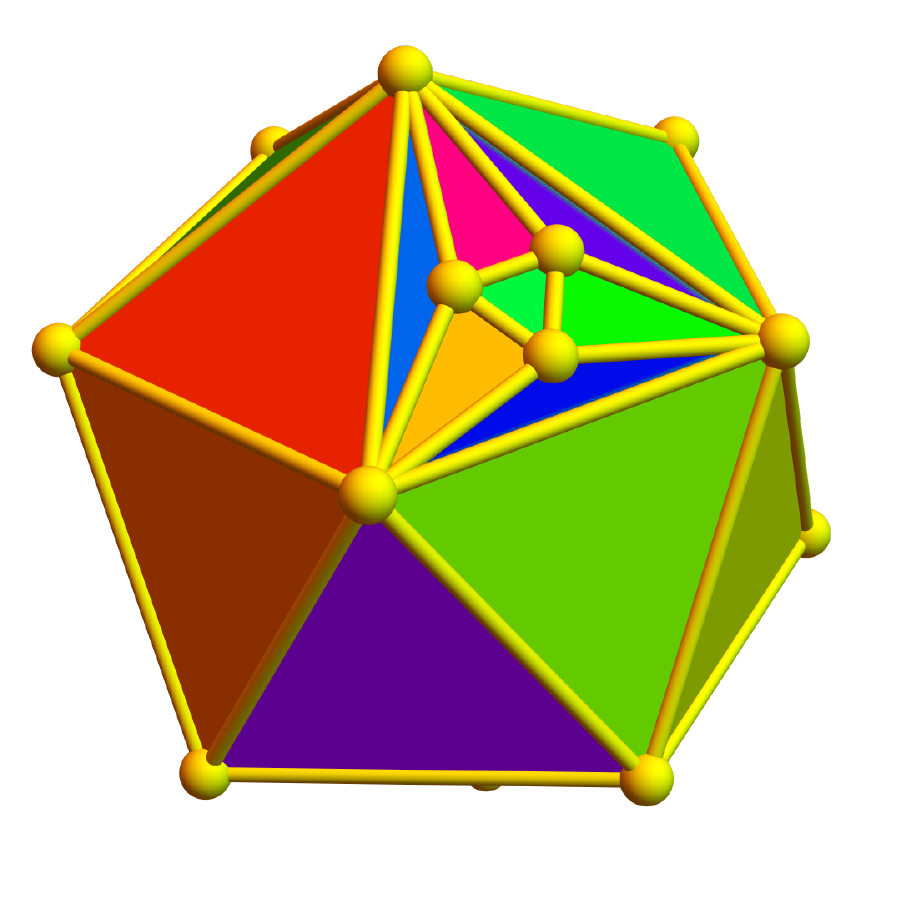}}
\label{Figure 0}
\caption{
Torsion is sensitive. Even for triangulations of spheres. To the 
left we see a 2-sphere with $f=(26,72,48)$ and $A(G)=13/24$ obtained by making random edge subdivisions
of an icosahedron. Edge subdivision preserves 2-spheres. To the
right we see a local refinement involving a triangle. This is of course no more a 2-sphere
as some unit spheres are not circular graphs anymore. The torsion $5/53$ has no relation 
any more to the $f$vector $(|V|,|E|,|V|)=(15, 39, 27)$. The Betti vector is $(1,0,2)$ as actually this
graph consists of two spheres glued together at a triangle and Euler Poincar\'e (which holds
for all graphs) looks $\chi(G) = 16-30+27=1-0+2$. 
}
\end{figure}

\begin{figure}[!htpb]
\scalebox{0.6}{\includegraphics{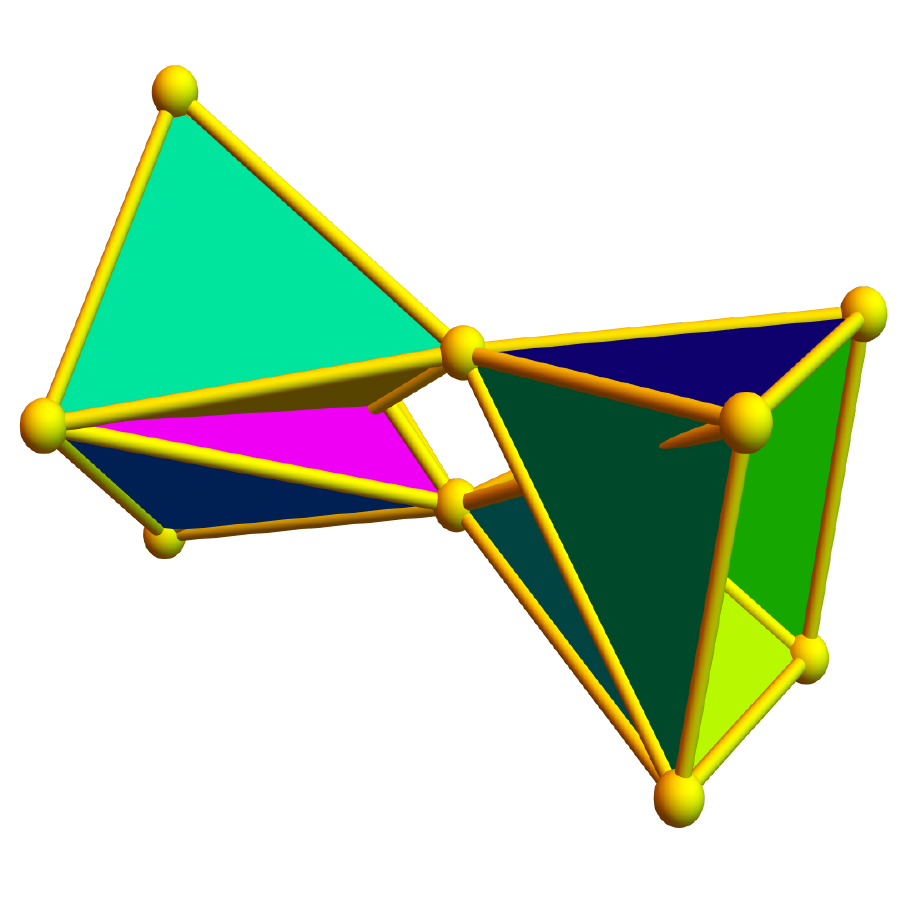}}
\label{Figure 0}
\caption{
A Dehn-Sommerville graph $G$ is a suspension of two disjoint circles 
$G= 2*C_4 \oplus S_0=(C_4 + C_4) \oplus S_0$. It is not a sphere like the octahedron $S_4 \oplus S_0$
This graph $G$ has the f-vector $f=(10, 24, 16)$, Betti vector $b=(1,1,2)$ and 
Euler characteristic $\chi(G) = 10-24+15=1-1+2=2$. The 
analytic torsion is $A(G) = 5/32$.  We experimentally see
that $A( 2C_{m} \oplus S_0) = (|V|/|V'|)4/m^2$. 
}
\end{figure}

\begin{comment}
Import["~/text3/graphgeometry/all.m"];
F[n_]:=Module[{},
s0=MyJoin[zerosphere, DisjointJoin[CycleGraph[n], CycleGraph[n]]];
AnalyticTorsion[s0]*Fvector[s0][[3]]/First[Fvector[s0]]];
Table[F[n],{n,4,10}]
\end{comment}

\paragraph{}
An other consequence of the {\bf key lemma}
$$  A(G) = {\bf SDet}(D) = \prod_{k \; {\rm even}} {\rm Det}(D_k)/\prod_{k \; {\rm odd}} {\rm Det}(D_k) $$
is that for all finite simple graphs, 
$|{\rm Det}(D)| A = \prod_{k \; {\rm even}} {\rm Det}(D_k)^2$ and
$|{\rm Det}(D)|/A = \prod_{k \; {\rm odd}} {\rm Det}(D_k)^2$ are both squares.
We have seen that experimentally in 2013, but at that time could not see why. 
Both the pseudo determinant of the Dirac operator ${\rm Det}(D)$ of a
graph as well as the super pseudo determinant $A(G) =  {\rm SDet}(D)$ of $D$
are {\bf interesting functionals} on the category of finite simple 
graphs. The general case is largely unexplored. Interesting problems are
to investigate the {\bf expectation} of $A(G)$ on {\bf Erd\"os-R\'enyi spaces} or
to look for the {\bf maxima} and {\bf minima} of torsion $A(G)$ on all 
graphs with $n$ vertices. 

\paragraph{}
We also look at torsion $A_2(G)$ on the {\bf Wu complex}. The set of {\bf $k$-forms}
are now functions on all pairs of intersecting simplices $(x,y)$ with 
${\rm dim}(x)+{\rm dim}(y)=k$. The exterior derivative is defined as 
as before 
$$ df(x,y) = \sum_{(z,w) \subset (x,y)} f(z,w) {\rm sign}((z,w),(x,y)) \; . $$
The definition of torsion is the same as for the usual {\bf Euler complex}. Just take the 
super determinant of the {\bf Wu Dirac operator} $D_{Wu} = (d+d^*)^2$, where $d$ is 
the exterior derivative of the Wu complex. The reason for the name ``Wu complex" is 
that instead of 
$$  {\rm SDet}(\lambda D_{Euler}) = \lambda^{2 \chi(G)} {\rm SDet}(D_{Euler}) $$ 
with {\bf Euler characteristic} $\chi(G)$, we have now 
$$ {\rm SDet}(\lambda D_{Wu}) = \lambda^{2 \omega(G)}  {\rm SDet}(D_{Wu})$$ 
where $\omega(G)$ is the {\bf Wu characteristic}. Note that for odd dimensional
discrete manifolds, where $\chi(G)=\omega(G)$ are both $0$, torsion 
is independent of the scale of the exterior derivative. 

\paragraph{}
For Wu characteristic, there is no analog simple formula in the contractible case. 
The formulas are similar for spheres: we have $A_2(G)=1/(f_{dd} f_{00})$ for odd-dimensional 
spheres and $A_2(G)=f_{00}/f_{dd}$ for even-dimensional spheres, 
where $f_{kl}$ is the {\bf f-matrix of $G$} counting intersections of $k$ and $l$ 
simplices in $G$. Unlike for torsion, the Wu torsion $A_2(G)$ is no more 
expressible in a simple way through the $f$-matrix of $G$, if $G$ is contractible. 
Also the Wu characteristic $\omega(G)$ is, unlike the Euler characteristic $\chi(G)$ 
not a homotopy invariant. See \cite{valuation,CohomologyWu}.

\section{Analytic torsion}

\paragraph{}
A {\bf finite abstract simplicial complex} 
is a finite set of non-empty sets closed under the operation of taking finite non-empty subsets. 
The set of vertex sets of complete sub-graphs of a graph $G=(V,E)$ defines such a finite 
abstract simplicial complex.  It is called the {\bf Whitney complex} or {\bf clique complex} of $G$.
In order to to calculus, equip first each simplex $x$ with an {\bf orientation}. This is a {\bf choice of 
basis} and irrelevant for all computations. We usually label the vertices $V$ of the 
graph with integers and take the natural order of these integers on each simplex like $y=(5,7,11)$ for
a triangle with vertices $5,7,11$. Given $x \subset y$, define ${\rm sign}(x,y)=1$
as the sign of the permutation which maps $x$ into the induced orientation of $y$. 
If $x$ is not a subset of $y$ define ${\rm sign}(x,y)=0$. For $x=(5,11)$ for 
example, ${\rm sign}(x,y)=-1$. For $x=(5,7)$ we would have ${\rm sign}(x,y)=1$. 

\paragraph{}
The {\bf incidence matrices} $d_k(x,y) = {\rm sign}(x,y)$ define a {\bf differential complex} with 
{\bf exterior derivative} $d(x,y)={\rm sign}(x,y)$, where $x,y$ are simplices.
If $G$ has $n$ simplices, then $d$ is a lower triangular $n \times n$ matrix
of the same size as the Dirac operator $d+d^*$ or the Hodge Laplacian $L=(d+d^*)^2$. 
If $x$ is a $(k+1)$-simplex, then the derivative is define as 
$df(x) = \sum_{y \subset x, {\rm dim}(y)=k}, {\rm sign}(x,y)$. As $d_k d_{k-1}=0$, 
one gets the $k$'th {\bf cohomology group}. It is the vector space $H_k(G)= {\rm ker}(d_k)/{\rm im}(d_{k-1})$.
Identified it with the {\bf harmonic k-forms}, the null space of $L_k$. 
linear algebra is the most elegant way to compute $H_k(G)$. The cohomology theory of such a finite complex
is historically the first and also the simplest. For a discrete manifold 
$G=(V,E)$, this simplicial cohomology is equivalent to {\bf de Rham cohomology} 
of a {\bf smooth geometric realization} of the complex.
We are here never interested in the continuum. The cohomology is defined for any 
network, and not only for discrete manifolds. 

\paragraph{}
As a consequence of $d_{k} d_{k-1}=0$, the {\bf Hodge Laplacian} 
$L=D^2=(d+d^*)^2$ of the graph $G=(V,E)$ decomposes into block matrices 
$L_k=d_k^* d_k + d_{k-1} d_{k-1}^*$, the {\bf $k$-form Laplacians} or {\bf Hodge blocks}.
The first block $L_0=d_0^* d_0$ is the {\bf Kirchhoff Laplacian} $B-A$, 
where $B$ is the diagonal vertex degree matrix and $A$ is the 
{\bf adjacency matrix} of the graph. The matrix $L_0$ is the discrete analog of
$\Delta = {\rm div} {\rm grad}$ in calculus and has been introduced
by Kirchhoff before Betti even defined the incidence matrices 
$d_k$. Dirac saw the power of writing a Laplacian $L$ as a square $L=D^2$
and Hodge related the spectrum of $L = L_0 \oplus L_1 \oplus \cdots \oplus L_d$ 
with cohomology: the {\bf Betti number} $b_k$ of a manifold is the nullity 
of the $L_k$ and the null-space of $L_k$ identifies with cohomology.

\paragraph{}
The {\bf pseudo determinant} $R = {\rm Det}(A)$ of a finite matrix $A$ is
defined as the product of the non-zero eigenvalues of $A$.
The number ${\rm Det}(L_0)$ is known to be the number of {\bf rooted spanning trees} 
in $G$. If $F_k=d_k + d_{k-1}^*$ is the $n \times f_k$ submatrix of the Dirac operator 
$D=d+d^*$, selecting the columns of $D$ belonging to $k$-dimensional simplices, 
then $L_k=F_k^T F_k$ is a block in the Hodge Laplacian $L$ and 
${\rm Det}(L_k) = \sum_{|P|=f_k-b_k} \det(F_k(P))^2$ sums over the 
squares of all $f_k-b_k$ minors of $F_k$, with 
$b_k={\rm dim}({\rm ker}(L_k))$. 

\paragraph{}
We have to use here the {\bf general Cauchy-Binet identity}
$$ {\rm Det}(F^T G) = \sum_{|P|=k(F,G)} {\rm det}(F) {\rm det}(G) $$
\cite{CauchyBinetKnill} 
which holds for arbitrary $(n \times m)$-matrices $F,G$ and generalizes the {\bf Cauchy-Binet theorem},
in which either $F^T F$ or $F F^T$ has full rank and where ${\rm Det}$ is replaced with the usual determinant 
${\rm det}$. While it is appears a small matter to go from determinants to pseudo determinants, the linear algebra
appearing in networks forces the more general situation: the matrices $L_k$ are in general singular,
and the Dirac blocks $D_k$ are almost always are singular, even if the cohomology group $H^k(G)$ should be trivial.
A special case of Cauchy-Binet for pseudo determinants is the {\bf Pythagorean identity}
$$  {\rm Det}(F^T F) = \sum_{|P|=k(F)} {\rm det}(F)^2  $$
which is of particular interest if ${\rm det}(F)$ is either $1$ or $-1$. In this case, the pseudo
determinant {\bf counts geometric objects} like {\bf trees} or {\bf rooted trees}.
Having seen this, it becomes apparent why torsion is  {\bf partition function} a functional 
that counting geometric objects for any graph $G$. Also ${\rm Det}(D)$ has this property but ${\rm Det}(D)$
counts while ${\rm SDet}(D)$ super counts. 

\paragraph{}
The squares ${\rm Det}(F_k(P))^2$ of the minors of the columns $F_k = d_k+d_{k-1}^*$ 
are integers. But they are in general larger than $1$ so that there is no simple 
{\bf geometric interpretation} of the pseudo determinant yet. Fortunately, also
$D_k = d_k^* d_{k}$ allow to express torsion. We have
$$  {\rm Det}(L_k) = {\rm Det}(D_k) {\rm Det}(D_{k-1})  \; . $$
(We learned this first from the lecture notes \cite{Bunke2015} and later saw it also in
\cite{GelfandKapranovZwlevinsky} or \cite{BurghelaFriedlanderKappelerMcDonald}.
This is extremely important as it clears up the rather mysterious definition of analytic 
torsion in terms of {\bf powers} of the matrices $L_k$. 
Torsion is much more natural as a super determinant which does not involve powers of 
the blocks. 

\paragraph{}
Because this is so crucial, let us reformulate it
the definition $A(G)$ is much more natural when seen in terms of the matrices $D_k$. 
Torsion 
$$ A(G) = \prod_k {\rm Det}(L_k)^{k (-1)^{k+1}}  \; . $$
can be identified with the {\bf super pseudo determinant of the Dirac operator}
$$ A(G) = \prod_{k} {\rm Det}(D_k)^{(-1)^k} \; . $$
It is a {\bf ``Fermionic" version} of the pseudo determinant of the Dirac operator
$$ {\rm Det}(D) = \pm \prod_{k} {\rm Det}(D_k) \;  $$
which as an orientation oblivious determinant and so has a more {\bf ``Bosonic"} nature.

\paragraph{}
While the pseudo determinant counts types of trees in a graph defined by the simplicial complex,
the super pseudo determinant and so the analytic torsion is the ratio of the number of even trees 
over the number of odd trees. For the 4-sphere, the {\bf cross polytope} with $f$-vector 
$f=(10, 40, 80, 80, 32)$ for example, we have a Dirac determinant ${\rm Det}(D)=2^{220} 3^{40} 5^{15}$, a number 
with $95$ digits while torsion $A(G) = 10/32=0.3125$ is small. For the first few cross polytopes $S^k$
(k-dimensional spheres) we have $A(S^0)=1,A(S^1)=16,A(S^2)=3/4,A(S^3)=128,A(S^4)=5/16,A(S^5)=768$.
$A(S^4)=5/16$ is by far not the minimum of $A$ on all graphs with $10$ vertices (graphs with the cohomology of 
bouquets of spheres have far lower torsion) but
$A(K_{5,5}) = 3906250$ might be the maximum of $A$ on all graphs with $10$ vertices. As we can not enumerate
all graphs with $10$ vertices, we made experiments with random graphs. 
% S={zerosphere,onesphere,twosphere,threesphere,foursphere,fivesphere}; f[x_]:=AnalyticTorsion[x]; Map[f,S] 
% min=1000; Do[ s=ErdoesRenyi[10,0.7];a=N[AnalyticTorsion[s]]; If[a<min,smin=s;min=a;Print[a]],{10000}]
% max=0;    Do[ s=ErdoesRenyi[10,0.4];a=N[AnalyticTorsion[s]]; If[a>max,smax=s;max=a;Print[a]],{10000}]

\paragraph{}
McKean and Singer paired the non-zero eigenvalues of 
even-form Laplacians $L_{2k}$ with 
the non-zero eigenvalues of the odd-form Laplacians $L_{2k+1}$. This allowed them 
to write the Euler characteristic ${\rm str}(1)$ as $\chi(G)={\rm str}(e^{-t L})$.
A consequence of this McKean-Singer symmetry is that for any finite simple graph, 
we have $1={\rm SDet}(L) = \prod_{k=0} {\rm Det}(L_k)^{(-1)^{k+1}}$.
The pseudo determinant filters out the complementing cohomology, which is also a
spectral part of $L$ but where the super symmetry between even and 
odd forms is broken if $\chi(G) \neq 0$. The McKean-Singer symmetry 
holds for all graphs and the formula ${\rm SDet}(L)=1$ becomes evident also 
from the identity ${\rm Det}(L_k) = {\rm Det}(D_k) {\rm Det}(D_{k-1})$. 
The super product of these telescopes to $1$. 

\paragraph{}
We define here the {\bf analytic torsion} of a graph $G$ as
$$   A(G) = \prod_{k=0} {\rm Det}(L_k)^{k (-1)^{k+1}}  \; . $$
It is the square of the definition usually taken in the continuum.
We do not take the square root, because we are in a combinatorial 
setting, where we are interested in the {\bf rational number} $A(G)$ and not
in the square root. We have $A(G) = {\rm SDet}(D)$, the super determinant
of the Dirac operator. Of course, also the super determinant of 
the exterior derivative $d$ makes sense, which is just the square root. 
For odd-dimensional manifolds, the classical notion
$\sqrt{A(M)}$ scales like Riemannian volume. Traditionally, the notion 
has been studied for odd-dimensional manifolds for which $\chi(G)=0$.
This is a situation, where $A(G)$ does not depend on how the scale of 
the exterior derivative. This follows from
$A(G,\lambda d) = \lambda^{2 \chi(G)} A(G,d)$. 

\paragraph{}
Torsion can been defined for arbitrary chain complexes as the super pseudo 
determinant of the Dirac operator $D=+d^*$ of the complex. 
In the graph case, a natural case is 
the {\bf Wu chain complex} rather than the Euler complex.
Other modifications can be done by deforming the exterior derivative. 
Examples are the nonlinear {\bf isospectral Lax deformation} 
\cite{IsospectralDirac2,IsospectralDirac} of $d$ or the 
{\bf Witten deformation}. As these deformations do not change the Laplacian,
torsion remains. However, we remind that under deformation, the Dirac operator
develops a diagonal part. If we go with the deformed $d+d^*$ (forgetting about
the dark matter part in the diagonal), then space expands using the Connes
formula and torsion will change because under a scaling $d \to \lambda d$, 
torsion changes like $A(G,\lambda d)  = \lambda^{2 \chi(G)} A(G,d)$, where 
$\chi(G)$ is the {\bf Euler characteristic}. In the case of the Wu differential complex,
it would scale like $A_2(G,\lambda d) = \omega^{2 \omega(G)} A_2(G,d)$,
where $\omega(G) = \sum_{x \sim y} \omega(x) \omega(y)$ is the 
{\bf Wu characteristic} of the graph $G$. \cite{valuation}.

\paragraph{}
Unlike in the continuum, the analytic torsion for graphs is interesting
also for even-dimensional discrete manifolds, like spheres. For 2-spheres, 
the formula for analytic torsion expresses the fact that the number of 
spanning trees in $G$ and its dual graph $G$ are the same, a fact which 
has been recognized already by Von Staudt of being
equivalent to the Euler-Gem Formula $|V|-|E|+|F|=2$ for $2$-spheres.

\paragraph{}
In order to render the combinatorially defined $A(G)$ 
a {\bf topological invariant} like making 
it invariant under Barycentric refinements, $A(G)$ needs to be scaled
by quantities given by the $f$-vector 
$f=(f_0,f_1, \dots, f_d)$. Already for circular graphs $C_n$, 
the number $A(C_n)=n^2$ agrees with the number of rooted 
spanning trees in the circle $C_n$. Indeed, in order to 
get a spanning tree, we can remove one of the $n$ edges and 
select one of the vertices as root. 

\paragraph{}
We experimentally
looked first to express $A(G)$ in terms of the $f$-vector 
$f=(f_0, \dots, f_d)$, where $f_k={\rm Tr}(1_k)$ is the number of 
$k$-dimensional simplices in $G$. 
Obviously, $A(G)$ can not be expressed in terms of the $f$ vector alone 
in general. The quantity changes also under local refinements. 
Already the one-dimensional case, where we understand
things pretty well, explains this.

\paragraph{}
Inductively, a graph $G$ is defined to be {\bf contractible}, if there is a vertex $x$ such that
$S(x)$ and $G-x$ are both contractible. The $1$-point graph is contractible.
A graphs is homotopic to $1$ if one can get using homotopy extensions (inverting
the process of removing a vertex with contractible unit sphere) and homotopy
reductions to $K_1$. The first result.

\begin{thm}[Torsion for graphs homotopic to 1]
For any $G=(V,E)$ homotopic to $1$ one has $A(G) = |V|$.
\end{thm}

\paragraph{}
The statement follows from a symmetry between objects
appearing in the even truncated Laplacian $H^+$ and the odd truncated
Laplacians $H^-$ belonging to {\bf rooted simplicial complexes}.
The quantity $|V| \det(H^+)$ has a geometric interpretation as even dimensional objects
while $\det(H^-)$ must have a geometric interpretation as odd dimensional objects.

\paragraph{}
Graphs homotopic to $1$ are the simplest from a homotopy point of view.
Torsion changes in a rather subtle way under homotopy transformations however
if the graph $G$ is not contractible. It turns out that the case of spheres is
managable. We need a manifold structure now which in graph theory means that
every unit sphere $S(x)$ of a vertex $x$ (the graph induced by the neighbors)
is a $(d-1)$-sphere. These are inductive definitions starting with the assumption
that the empty graph is the $-1$-spheres.

\paragraph{}
We see here for odd-dimensional
discrete $d$-spheres that $A(G) = f_0(G) f_d(G)$ so that $\sqrt{A(G)}$
is in this spherical situation a volume, namely the geometric mean between
the vertex cardinality of $G$ and
the vertex cardinality of the dual graph $\hat{G}$. For even
dimensional spheres, where we conjecture $A(G)=f_0(G)/f_d(G)$,
the quantity goes to zero under Barycentric refinements which in line
with the fact that analytic torsion of even dimensional manifolds is
classically zero.

\section{Examples}

\paragraph{}
A $0$-dimensional graph has no edges and is a discrete set
$V$ of points. The analytic torsion of such a graph $G$ is 
just $f_0(G)=|V|$, the number of vertices. In general,
if $G=H \cup K$ is a disjoint union of graphs $H,K$, then 
$A(G) = A(H) + A(K)$. Already in the one-dimensional case
we can see that $A$ is {\bf not} a {\bf valuation}
$A(G)=A(H)+A(K)-A(H \cap K)$. To motivate the following
proposition, let us look at the question how we would have to 
scale $A(G)$ so that it becomes a valuation. 

\paragraph{}
Take two linear graphs $H,K$ of length $n$. They both have 
$A(H)=A(K)=n+1$. Glue them together to get a circle $G$ of 
length $2n$ with $A(G) =4n^2$ (the number of rooted spanning
trees). If we see $H,K \subset G$ with
$H \cap K = P_2$, the 2-point graph with $A(P_2)=2$, we
would have to scale $A(G)/(f_0 f_1)$ for circles and $A(G)/f_0$
for intervals in order to get something which satisfies the 
{\bf valuation property} $A(G)=A(H)+A(K)-A(H \cap K)$.

\begin{propo}
If $G$ is triangle free, then $A(G)$ is the number of rooted spanning trees in $G$. 
\end{propo}
\begin{proof}
If $G$ is triangle free, then the maximal dimension is $1$ and 
$$  A(G)= \frac{{\rm Det}(L_1)^1}{{\rm Det}(L_0)^0}  = \frac{{\rm Det}(L_1)^2}{{\rm Det}(L_0)}  \; . $$
By McKean-Singer symmetry ${\rm Det}(L_1) = {\rm Det}(L_0)$ and we have $A(G)={\rm Det}(L_0)$. 
By the matrix tree theorem, this is number of rooted spanning trees in $G$. 
\end{proof}

\paragraph{}
To illustrate this, we look at {\bf cacti graphs} of genus $1$. 
These graphs are obtained from a
circular graph $C$ by attaching arbitrary many trees.
We have just seen that $A(G) = {\rm Det}(L_0)$ is the
number of rooted spanning trees in $G$. As each of these trees is determined by
removing an edge in $C$ and selecting out a vertex in $V$, the number
$|V| |E(C)|$ is the number of rooted spanning trees in $G$. We therefore have 
$A(G)=|V| |E(C)|$. 

\paragraph{}
More generally, if we have a {\bf bouquet of $1$-spheres} 
obtained by a wedge sum of $1$-spheres then the Matrix tree theorem 
again shows that $A(G)$ depends on the product of the lengths 
of the fundamental cycles. We see from this example
already that $A(G)$ not only can depend on the $f$-vector 
$f=(f_0,f_1,f_2, \dots, f_d)$ of $G$ but that it also depends on 
{\bf torsion elements}, like the size of the circular non-contractible
parts of $G$. 

\begin{coro}
For a triangle-free graph obtained by attaching finitely many trees to 
a bouquet of spheres $C_{n_1} \wedge C_{n_2} \cdots \wedge C_{n_k}$,
we have $A(G)=|V| \prod_{k} |E(C_k)|$.
\end{coro}

\paragraph{}
This generalizes to the situation of graphs $G$ with fundamental group $F_n$. 
Take a bouquet of one dimensional spheres $C_k$ and then attach arbitrary many 
trees. In that case, the number of spanning trees is known. 

\paragraph{}
For complete bipartite graphs $K_{2,n}$ for which the 
Betti vector is $b= (1,n-1)$ we see 
$A(K_{2,n}) = 2^n (n+3) (n+1)$.
For $K_{3,n}$ with $b=(1,2n-2)$ we see
$A(K_{3,n}) = 3^n (n+4) (n+1)^2$. In general we can show that
the number of rooted spanning trees is

\begin{coro}
$$ A(K_{k,n}) = n^{k-1} k^{n-1} (n+k) \; . $$
\end{coro}

For $K_{k,l,n}$ the maximal dimension is $2$, 
the Betti vector $b=(1,0,(k-1)(l-1)(n-1))$. 

\paragraph{}
As we have just seen in one dimensions, 
the quantity $A(G)$ involves not only the $f$-vector of $G$ but also
involves the volumes of generators of a homology group $\pi_1(G)$. 
This also is the case in two dimensions. Let us look for a bouquet of $m$ 
two-spheres $S_1,\dots,S_m$ to which an arbitrary number of 
trees has been attached. For tetrahedra-free graphs we have in general
$A(G)={\rm Det}(L_1)/{\rm det}(L_2)^2={\rm Det}(L_0)/{\rm Det}(L_2)$.
Now $f_0 {\rm Det}(L_2)$ is the number of rooted trees in $G$ and 
$f_2 {\rm Det}(L_0)$ is the number of rooted trees in the 
{\bf dual graph} $\hat{G}$ in which the triangles are the vertices and
two triangles are connected if they intersect in an edge. 
Remember that a $2$-sphere is a finite simple graph for which every unit
sphere is a circular graph. The following result has been known 
since the mid 19th century for planar graphs. We use it in the special 
case of 2-spheres $G$, where the dual graph is a triangle 
free graph. If $f(G)=(|V|,|E|,|F|)$ then 
$|V|-|E|+|F|=2$ by {\bf Euler's Gem} formula and $f(\hat{G}) = (|F|,|E|)$. 

\begin{lemma}[Maze lemma of van Staudt]
If $G$ is a $2$-sphere with dual graph $\hat{G}$ then the number of rooted
spanning trees in $G$ and $\hat{G}$ are the same. 
\end{lemma}
\begin{proof}
Draw both graphs $G$ and $\hat{G}$ in the same picture, where both graphs
have the same edge sets. A spanning tree $T$ in $G$ is a {\bf connected maze} 
in $G$. The complement of $T$ in $E$ defines a complement maze, a spanning
tree in $\hat{G}$. Now use that every
spanning tree in a graph $G=(V,E)$ has $|V|-1$ edges. 
Since $|E| = [ |V(G)|-1]  + [ |V(\hat{G})|-1] = |V|-1 + |F|-1$, 
this immediately gives $|V|-|E|+|F|=2$. 
\end{proof} 

\paragraph{}
This implies that if $G$ is a 2-dimensional connected discrete manifold
(a graph for which every unit sphere is a circular graph with 4 or more
elements), then the number of spanning 
trees in $G$ is the number of spanning trees in $\hat{G}$ if 
and only if $G$ is a $2$-sphere. 
The above lemma also gives a formula in two dimensions which we believe 
to hold in general $A(G)=f_0(G)/f_{d}(G)$ for spheres of even dimensions $d$. 

\begin{coro}
If $G$ is a $2$-sphere, then $A(G)=f_0(G)/f_2(G)$. 
\end{coro}
\begin{proof}
$A(G)={\rm Det}(L_1)/{\rm det}(L_2)^2={\rm Det}(L_0)/{\rm Det}(L_2)$. 
The statement follows from 
$$ {\rm Det}(L_0)/{\rm Det}(L_2) = f_0(G)/f_2(G) $$ 
which expresses that the number of spanning trees in 
$G$ and the number of spanning trees in the dual graph are the same. 
\end{proof}

\begin{figure}[!htpb]
\scalebox{0.6}{\includegraphics{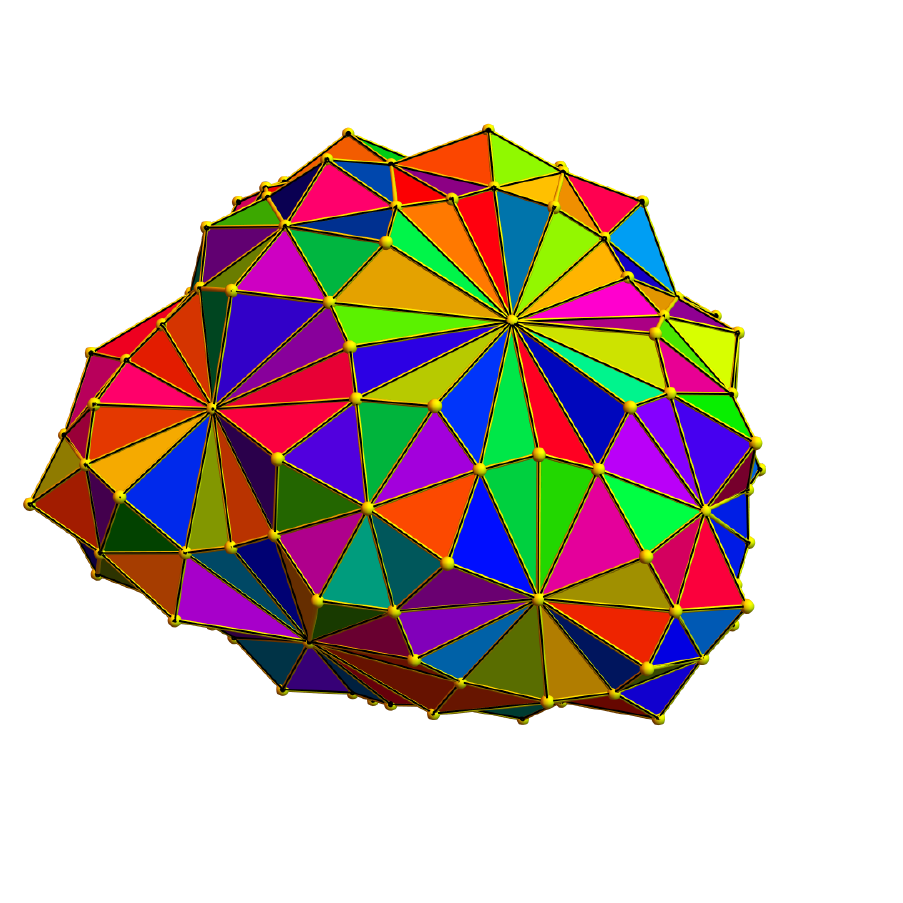}}
\scalebox{0.6}{\includegraphics{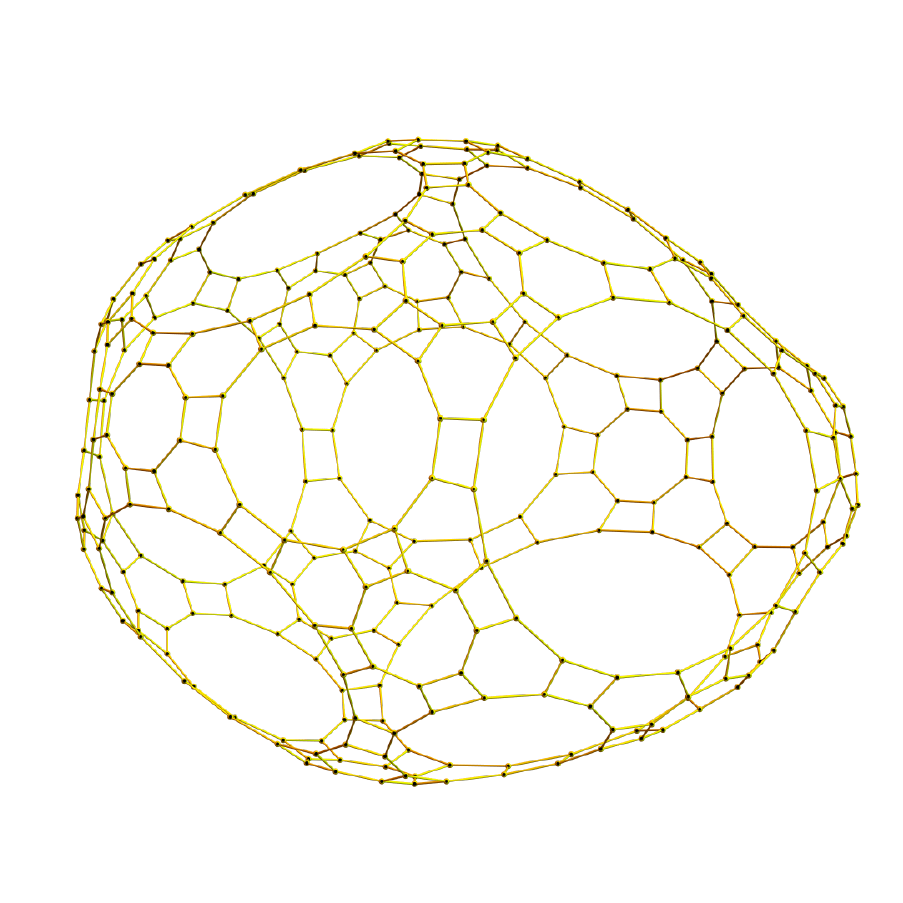}}
\label{Figure 0}
\caption{
A random $2$-sphere with $f$-vector $f=(164,482,324)$ and 
its dual sphere $\hat{G}$ which is triangle free as by definition,
every unit sphere has $4$ or more elements. We have 
$f_0(G)/f_2(G) = 164/324$ which is also the ratio of the 
number of rooted trees in $G$ 
%\begin{tiny}
%793843726541607936245506321785580914940895220397666883540916245031301788759470811022962387271651253881871007744
%\end{tiny}
and the number of rooted trees in $\hat{G}$
%\begin{tiny}
%1568325410972444947216732001576391563663719825663683355288151606037449875354076480313657399243993940595891503104
%\end{tiny}
}
\end{figure}

\paragraph{}
The partition of the edge set $E$ into two complementary trees has
been used in 1847 by Von Staudt \cite{Staudt1847} to prove the Euler 
polyhedron formula $|V|-|E|+|\hat{V}|= 2$ for planar graphs $(V,E)$. 
See also \cite{dehnsommervillegaussbonnet}. A spanning tree of $G$ and the 
dual spanning tree of $\hat{G}$ partition $E$ into two sets 
of $|V|-$ and $|\hat{V}|-1 = |F|-1$ elements so that 
$|E|=(|V|-1) + (|F|-1)$ which is the Euler formula. 

\paragraph{}
For example, if $G$ is the octahedron graph with $f=(6,12,8)$ 
then the dual graph $\hat{G}$ is the cube graph. 
We have ${\rm Det}(L_0(G))=2304$ and 
${\rm Det}(L_0(\hat{G}))=3072$ and $3072/2304 = 8/6$. 
Indeed ${\rm Det}(L_0(G))/8=384$ is the number of spanning trees in $G$
and ${\rm Det}(L_0(\hat{G}))/6=384$ is also the number of spanning 
trees in $\hat{G}$. 

\begin{coro}
If $G$ is a bouquet of 2-spheres $S_{n_1} \wedge \cdots \wedge S_{n_k}$ 
with arbitrary number of trees attached. Then 
$A(G)=f_0(G)/\prod_{j} f_2(S_{n_j})$. 
\end{coro}

\begin{proof}
Again, the ratio ${\rm Det}(L_0)/{\rm Det}(L_2)$ is the number of spanning 
trees in $G$ divided by the number of spanning trees in the dual 
graph $\hat{G}$. The dual graph has $\prod_{j} f_2(S_{n_j})$ 
spanning trees because the dual graphs $\hat{S_k}$ of the individual 
spheres $S_k$ are disjoint. 
\end{proof} 

\paragraph{}
A 2-dimensional manifold $G$ is a graph for which all unit spheres are
circular graphs of length $4$ or more.
In general, for 2-dimensional manifolds, 
${\rm Det}(L_2)$ is the number of rooted spanning trees in $\hat{G}$
and ${\rm Det}(L_2)/f_2(G)$ is the number of spanning trees in $\hat{G}$. 
The relation between the number of spanning trees in $G$ and 
$\hat{G}$ is not a topological invariant and very much 
depends on the topology as well as the metric realization. 

\paragraph{}
Lets look at some discrete manifolds with boundary. The case of an annulus 
is interesting as it is the simplest two-dimensional example with a non-trivial 
fundamental group. 

\begin{figure}[!htpb]
\scalebox{0.6}{\includegraphics{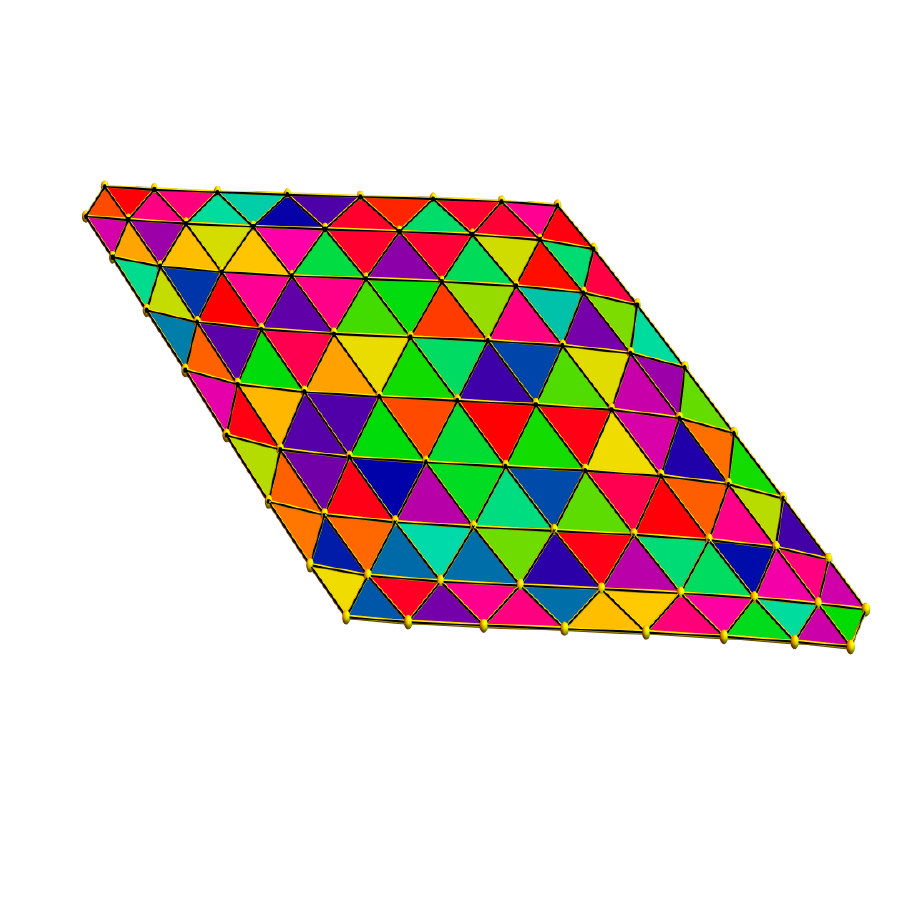}}
\scalebox{0.6}{\includegraphics{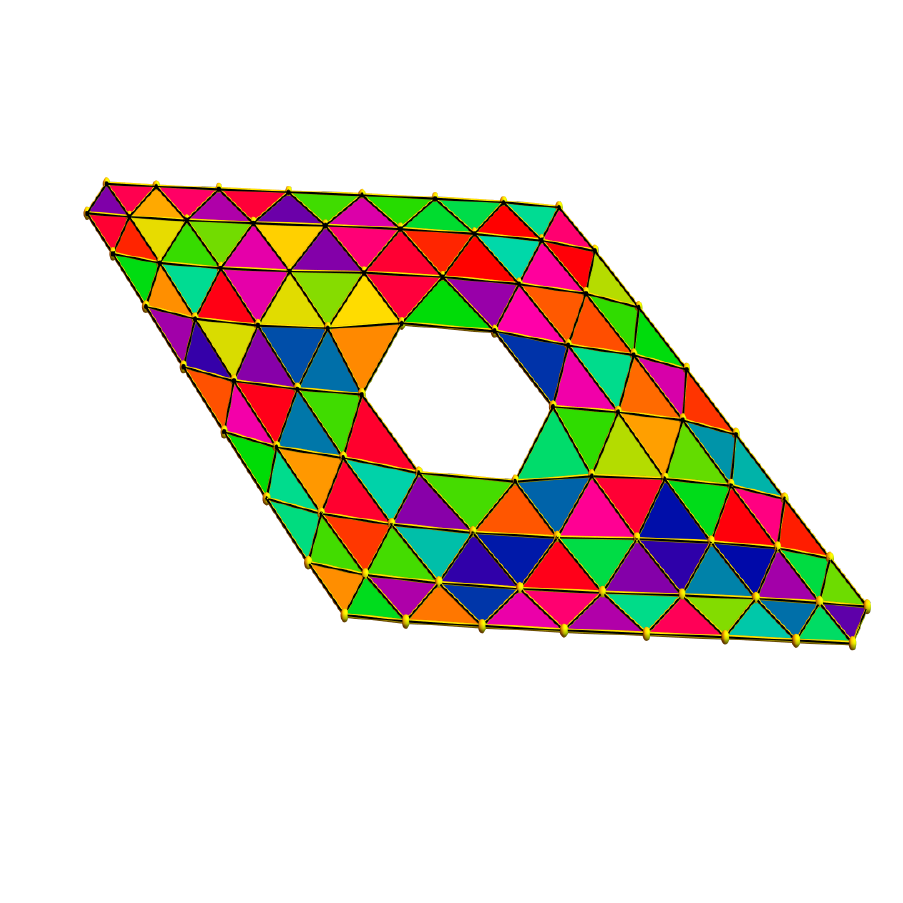}}
\label{Figure 4}
\caption{
To the left, we see a 2-ball with $f$-vector $(79, 204, 126)$
and Betti vector $(1,0,0)$.  It is a contractible graph for which we know
$A(G) = |V(G)|$ is the number of vertices. When drilling 
a hole, we get $b=(1,1,0)$ and torsion becomes complicated. 
In this case ${\rm Det}(D_0)= 294612705609082473864324666564747553420869058560$
is the number of rooted trees and 
${\rm Det}(D_1)=1971666443643036305928930398555703839352389632$.
}
\end{figure}

\paragraph{}
{\bf Example:} for the {\bf dunce hat} $G$ with $|V|=17$ vertices, where $A(G)=|V|$,
we have  \\
\begin{tiny}
${\rm det}(L)*17^2=14305210914701770615026899229224961147054878289*17^2=1425929411781^4$.
\end{tiny}

\paragraph{}
For complete graphs $G=K_n$, where $A(G)=|V(G)|=n$, we have
$\sqrt{{\rm det}(D)/|V|}=2^{n-2}-1$ also
because all eigenvalues of $D$ are either
$\pm \sqrt{n}$ or $0$.

\paragraph{}
For a $3$-spheres, we have $A(G)=|V'|/|V|$. The torsion is 
$$ A(G)=({\rm Det}(L_1)^2 {\rm Det}(L_3)^4)/({\rm det}(L_0)^1 {\rm det}(L_2)^3) \; . $$
This simplifies with McKean-Singer to 
$$   A(G) = \frac{{\rm Det}(L_3)^2 {\rm Det}(L_0)^0}{{\rm det}(L_2)^1} \; . $$
We have ${\rm Det}(L_0)/f_0$ as the number spanning trees in $G$
and ${\rm Det}(L_3)/f_3$ as the number of spanning trees in $\hat{G}$. 

\begin{comment}
\paragraph{}
Here is an experiment: rewrite 
${\rm Det}(L_0) {\rm Det}(L_1) {\rm Det}(L_3)^3/({\rm Det}(L_0) {\rm det}(L_2)^2) 
= f_0 f_3$ to see that the product of spanning trees of $G$ and
spanning trees of $\hat{G}$ is 
$ab = Det(L_0) Det(L_2)^2/Det(L_1) Det(L_3)^2 
= (Det(L_0)/Det(L_3)^2) (Det(L_2)^2/Det(L_1) = c*d$.
We see that the claim for 3-spheres is that the product $ab$ of the
number of spanning trees in $G$ and $\hat{G}$ is the same
than the product $cd$ of the above determinant quantities. 
Noticing that $c/a=b/d$ are integers suggests that
$(Det(L_0)/Det(L_3)^2)$ is a multiple $K$ of the number of 
spanning trees in $G$ and that also 
$(Det(L_2)^2/Det(L_1)$ is the same multiple $K$ of 
the number of spanning trees in $\hat{G}$. 
From the just done exploration, we define $K=c/a = 
[{\rm Det}(L_0)/f_0]/[(Det(L_0)/Det(L_3)^2)]$ which simplifies to 
$$ K = Det(L_3)^2/f_0 $$
Also $b/d = {\rm Det}(L_3) f_0 f_3 $
\end{comment} 

\begin{conjecture}
$A(G) = \prod_{i=0}^m f_2(S_i)/(m f_0(G))$. 
\end{conjecture}

\begin{figure}[!htpb]
\scalebox{0.5}{\includegraphics{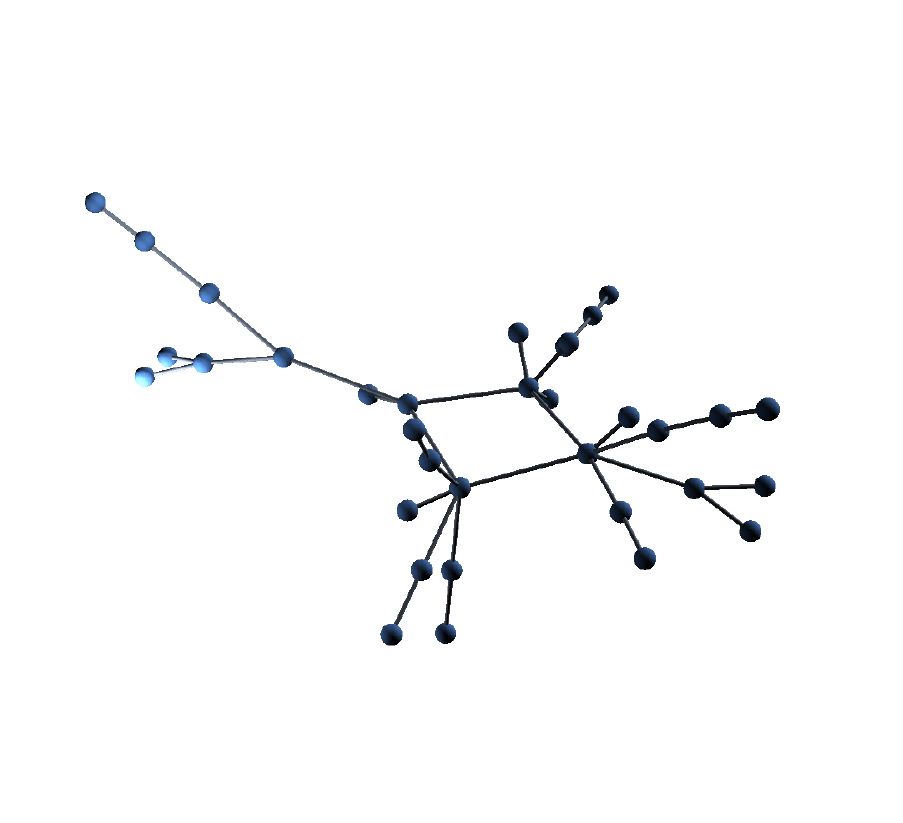}}
\scalebox{0.5}{\includegraphics{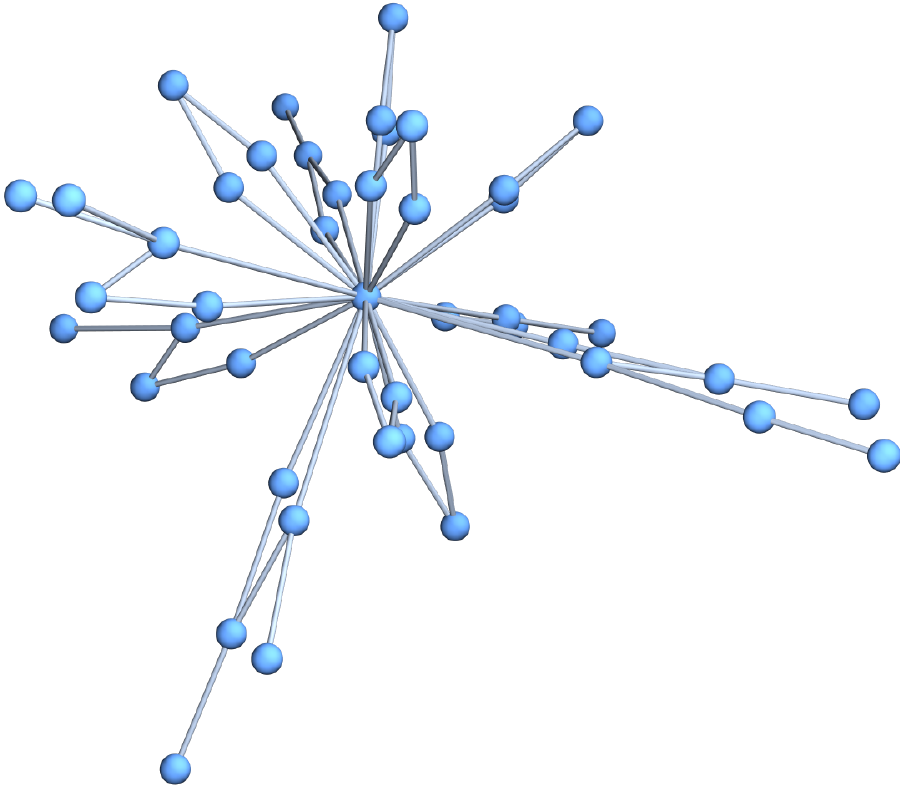}}
\label{Figure 1}
\caption{
A cactus graph of genus $1$ is a circular graph with trees attached. 
The analytic torsion is $A(G) = |V(G)| |E(C)|)$, where $C$ is the 
single circular sub-graph of $G$. 
To the right we see a one-dimensional genus 12 graph with 
fundamental group $F_{12}$, where $A(G) = |V(G)| \prod_i |E(C_i)|$
and $V(C_i)$ are the lengths of the loops $C_i$ generating the fundamental group. 
}
\end{figure}

\paragraph{}
Given an integer $n$, we can ask for which $p$ is the expectation
$E_{n,p}[A(G)]$ of Torsion on the probability space $\Omega_{n,p}$ of
Erd\"os-R\'enyi Graphs with $n$ vertices. This is difficult as we do
not know already the expected Betti numbers $E_{n,p}[b_k(G)]$.
For now, we can just make experiments.

\begin{figure}[!htpb]
\scalebox{0.7}{\includegraphics{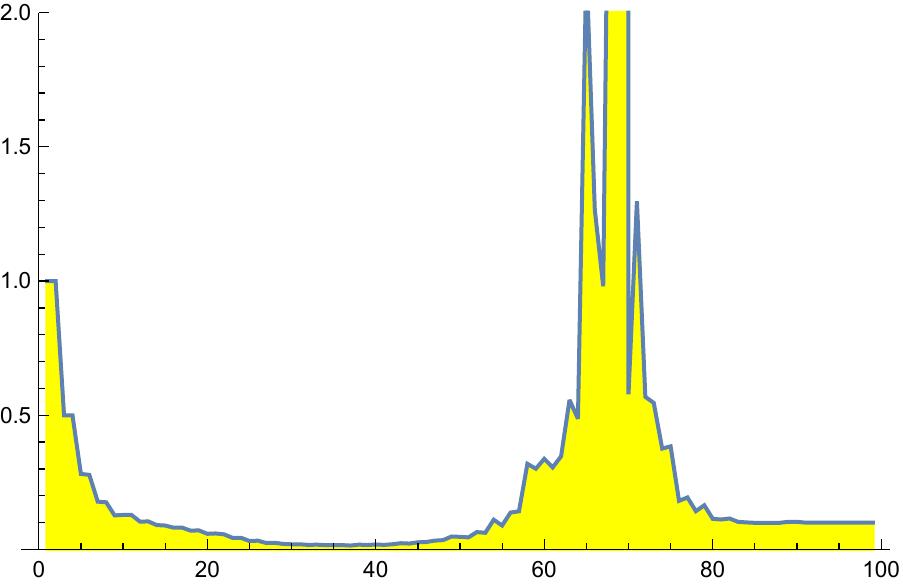}}
\label{Figure 3}
\caption{
A numerical computation of the expectation of analytic
torsion on $E(10,p)$ from $p=0$ to $p=1$. We computed
}
\end{figure}

\paragraph{}
The graph complement of cyclic graphs $C_n$ is interesting for various reasons
\cite{GraphComplements}. We measure for $A(C_n^c)$ the numbers \\
$(1,1,1,4,25,50,\frac{49}{5},\frac{4}{5},\frac{75}{196}, \frac{100}{21},\frac{1452}{7}, 
\frac{39204}{49}, \frac{169}{4}, \frac{49}{121},\frac{1620}{20449})$.
and for the complements of linear path graphs  \\
$(1,1,2,4,\frac{55}{3},\frac{156}{11},7,\frac{104}{85},\frac{45}{19},10,\frac{253}{2},\frac{1260}{17}, 13, 
\frac{931}{1334})$. We understand here the cases when $L_n^c$ is contractible in which case we have $A(L_n^c) = n$. 
We have not yet figured out whether there is a formula for all these rational numbers. 

\begin{figure}[!htpb]
\scalebox{0.6}{\includegraphics{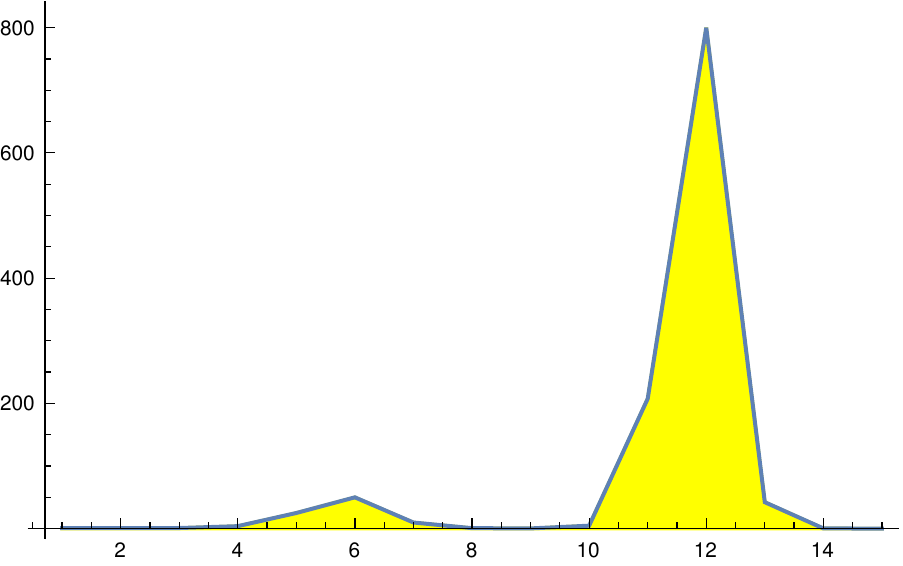}}
\scalebox{0.6}{\includegraphics{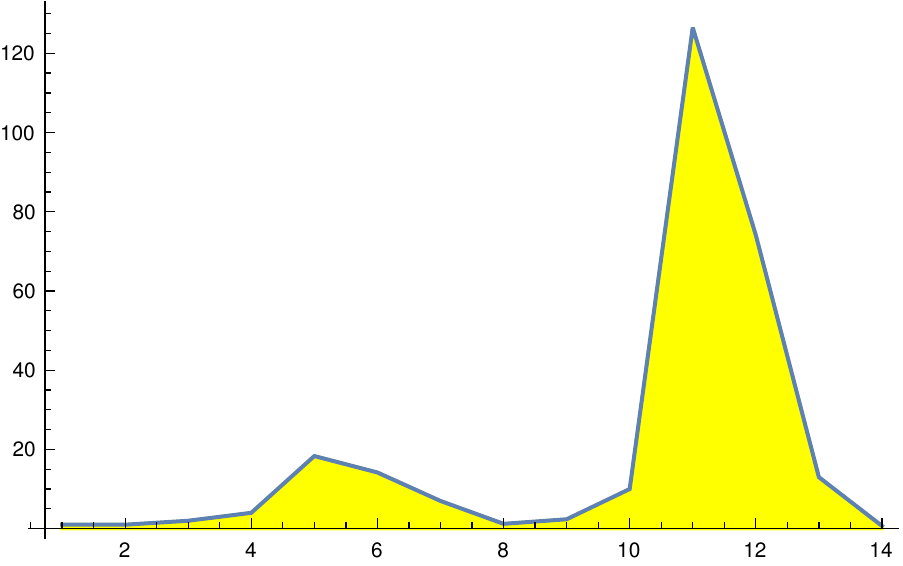}}
\label{Figure 4}
\caption{
To the left we see the first 15 values of analytic torsion of the 
complements of circular graphs $C_n$. To the right the first
14 values of analytic torsion of the complements of linear path graphs $L_n$. 
}
\end{figure}

\section{Remarks}

\paragraph{}
Looking at graphs is almost equivalent than to look for 
{\bf finite abstract simplicial complexes}. 
Any graph defines a simplicial complex, the Whitney complex and
any finite abstract simplicial complex $G$ defines a graph 
$(V,E)=(G,\{ (x,y), x \subset y or y \subset x\} )$
given by incidence. The language of graphs 
is more approachable as we are familiar with graphs and
networks like street or subway networks or family trees early on, while 
simplicial complexses already involve the concept of sets.
Simplicial complexes are amazing too \cite{AmazingWorld}.

\paragraph{}
$A(G)$ is the squared {\bf super determinant} ${\rm SDet}(D) = {\rm SDet}(d)^2$ of 
the chain complex. The notion of a determinant of a chain complex has been  put forward 
in \cite{GelfandKapranovZwlevinsky}.
Since $A(G)$ is a combinatorial notion and defined for any finite simple graph $G$
and not only if $G$ is a discrete manifold or evaluated in some limiting cases,
we do not take the root $\sqrt{A(G)}$. This allows us to stay rational.

\paragraph{}
The manifold case needs analysis: 
the {\bf Ray-Singer determinant} of a manifold defined by analytically
continued {\bf Minakshisundaram-Pleijel zeta functions} which can 
be defined using {\bf heat kernels}. This requires assumptions or the 
``magic hand waving" of a jedi assuring that the pole singularities which 
might occur for the individual zeta functions cancel at $0$ \cite{Witten2021}.
At least initially, Ray-Singer had to make strong assumptions on the cohomology
of the manifold. 

\paragraph{}
In the case of the circle $M=\mathbb{T}^1=\mathbb{R}/(2 \pi \mathbb{Z})$ 
for example, where $L_0=L_1=-d^2/dx^2$, we have the eigenvalues $\lambda_n = n$
of the Dirac operator $i \frac{d}{dx}$ with eigenvectors $e^{i n x}$. 
for $n \in \mathbb{Z}$. This shows 
$\zeta_M(s)=2 \zeta(2s)$ with the {\bf Riemann zeta function} 
$$  \zeta(s) = \sum_{n> 0} n^{-s}   \; . $$
we have included the factor $2$ is there because the spectrum of $D^2$ on $0$-forms
or $D^2$ on $1$ forms are both doubled. 
We have $\zeta_M'(0)=4 \zeta'(0)=-2 \log(2\pi)$ and 
${\rm Det}_M(L_0) = {\rm Det}_M(L_1) =  e^{-\zeta_M'(0)} = (2\pi)^2$. 
Therefore, $A_{RS}(M)=$  $\sqrt{{\rm Det}(L_0)^0/{\rm Det}(L_1)^1}$ $=(2\pi)$. 
In order not to confuse the torsion notions, we call $A_{RS}(M)$  the 
square root of $A(M)$. This is now the volume (circumference) of the circle. 
The formula $A(M) = (2\pi)^2$ for the squared torsion is the continuum the
equivalent of the formula $A(C_n) = n^2$ for {\bf circular graphs} $C_n$. 
This example illustrates already in the simplest possible manifold case that 
analytic continuation issues occur. 

\paragraph{}
The Hodge Laplacian $L$ contains spectral information and much is certainly 
still hidden. The Betti numbers $b_k(G)$ are the dimensions of the space 
of Harmonic $k$-forms ${\rm ker}(L_k)$ which according to Hodge are the 
$k$'th cohomology groups. (The observation that everything works just
using linear algebra in the discrete seems first have been done by 
\cite{Eckmann1}). The fact that the spectral data define analytic data in the 
form of zeta functions and so lots of other quantities will certainly lead
to more interesting quantities like the {\bf roots of the zeta function}. 

\paragraph{}
The McKean-Singer symmetry ${\rm str}(L^k)=0$ for $k>0$ implies 
${\rm str}(e^{-t L}) = {\rm str}(1) = \chi(G)$,
pairs the {\bf non-zero eigenvalues} of even and odd forms and 
immediately implies the {\bf McKean Singer relation}
$\prod_k {\rm Det}(L_k)^{(-1)^k}=1$ which is a consequence of the fact
that the {\bf Dirac operator} $D=d+d^*$
produces this {\bf super-symmetry relation} or spectral symmetry 
between even and odd forms. 
The simplest kind of super symmetry is the existence of a self-adjoint $P$ 
with $P^2=1$ and $DP=-PD$. This implies that if $D v = \lambda v$
we have $D Pv = -P Dv=- P \lambda v= - \lambda P v$ so that $P$ pairs
eigenvectors. The matrix $D$ maps non-Harmonic eigenvectors of $L$ on 
even forms to eigenvectors on odd forms and vice versa. 

\paragraph{}
If $G$ is contractible, there is
only one harmonic $0$-form, the constant function. Let $H$ be the matrix
obtained from $L$ by deleting the first row and first column.
In the contractible case $H$ is invertible. We can write $H$ as a direct
sum $H^+ \oplus H^-$ of the even-dimensional blocks and odd-dimensional 
blocks. The McKean-Singer relation can be rephrased using usual determinants
$$   {\rm det}(H^+) = {\rm det}(H^-) $$

\paragraph{}
By writing the determinant as a sum of permutations, we hope to have
a pairing between these elements, if the graph $G$ is contractible. 
How would such a pairing look like? Can we pair every permutation of
the $(n-1)/2$ simplices $H^+$ with a permutation of the 
$(n-1)/2$ simplices of $H^-$. If that would be the case, we could try
to get a pairing between rooted versions and get analytic torsion.

\paragraph{}
We also expect that we can in the case of the truncated $D$ (called $M$),
write $M=E^2$ for some complex $E$. 
Longer shot: by building up the complex, we can always pair a positive 
new eigenvalue with an even simplex and one with an odd simplex. This
pairing of simplices can give us a pairing of permutations. Now we 
would have to show that the product of the values are the same. 

\paragraph{}
An other corollary is that if $G$ is triangle free, then 
$Det(L_0) = Det(L_1)$ and since $Det(L^2)=Det(L_0) Det(L_1)=Det(D^2)$
we have $Det(D) = Det(L_0)$ counts the number of rooted trees in G
and $Det(D)/|V| =1$ counts the number of trees in $G$. 
In general, in the contractible case, $Det(D)/|V|$ is a square. 

\paragraph{}
Remark. For practical reasons, we
always assume that a sphere has the property that removing one point makes it
contractible. (Unlike ``homotopic to 1`` which is an NP complete task,
``contractible" can be checked fast. Inductively, a graph is called {\bf contractible},
if there exists a vertex $x$ such that $S(x)$ and $G-x$ are both contractible.
The induction assumption is that $K_1=1$ is contractible.)

\paragraph{}
In general, if $G$ has a non-trivial cohomology or even non-trivial homotopy 
groups, things
get more complicated. Analytic torsion $A(G)$ is not invariant under
Barycentric refinement, nor invariant under homotopy deformations even if rescaled.
The first case to look at are deformed spheres. We see there that the formula
for spheres has to be modified. Only the volume of the original
underlying d-sphere $H$ matters. It is somehow a volume of a cohomology class.

\paragraph{}
For a homotopy deformed $d$-sphere $G$ coming from an actual $d$-sphere $H$,
the volume of $H$ matters. Here are small dimensional examples:
For a homotopy deformed circle $G$ coming from a sphere $H=C_n$
that $A(G)/(|V(G)|*|E(H)| = 1$.
But this formula is already false if we add additional triangles.

\paragraph{}
In the $2$-dimensional case, we have invariance if we add lower
dimensional parts. But adding three dimensional part directly does not
work.  For a homotopy deformed 2-sphere $G$ coming from a sphere $H$
that $A(G)/(|V(G)|/|H|)=1$.

\paragraph{}
We experimented also with deformations of 3-spheres:
we see that for a homotopy deformed 3-sphere $G$ coming from a 3-sphere $H$
we have $A(G)/(V(G) |H|) = 1$. But only as long as the deformed part intersects
in lower dimensional parts.

\begin{figure}[!htpb]
\scalebox{0.8}{\includegraphics{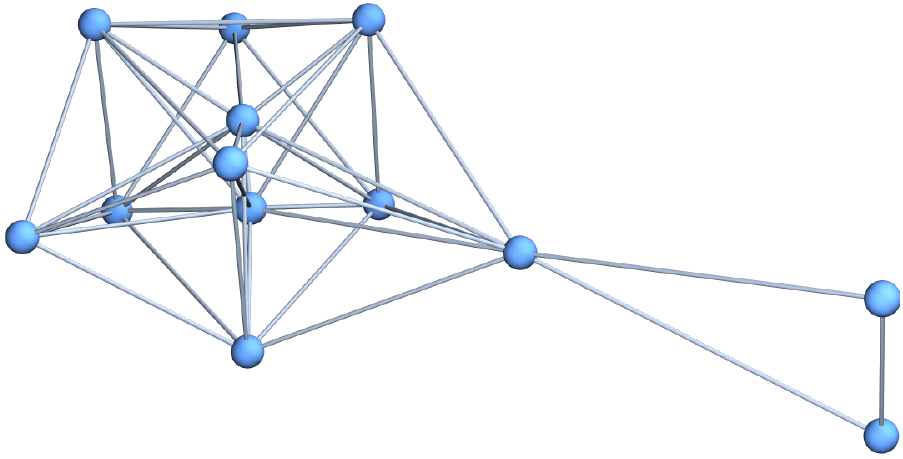}}
\label{Figure 1}
\caption{
A homotopy deformed 3-sphere $G$ which is obtained from
3-sphere $H$ with volume $|H|=28$. We have $A(G)=|V(G)| |H|$.
}
\end{figure}

\paragraph{}
How do we get in general torsion $A(G)$ from combinatorial data of $G$?
In general this is not yet solved.  Data can include the {\bf f-vector} 
as well as the lengths of minimal non-contractible spheres embedded in $G$. 
The following computation illustrates the effect of the fundamental group:
if $G$ has no triangles and is homotopic to a bouquet of spheres 
$H$ with $n$ loops, then $A(G)=|V(G)| |E(H)|^n/(n+1)$. 

\paragraph{}
We have tried to get relations for various products or sums or graph operations. 
We are also in interested 
in how $A(G)$ depends on algebraic operations. An easy case is  when
$G,H$ are disjoint union $G+H$ of graphs. In that case
$$   A(G + H) = A(G) A(H) \; . $$

\paragraph{}
We have seen that for some triangulations of the 2-torus, we have
$A(G) = 1/6$ but that for others, it is not.

\paragraph{}
Interestingly, there is also the following observation which we can 
not prove yet for the {\bf Shannon product} $C_n * C_m$ of two circles. 

\begin{conjecture}
$A(C_n * C_m) = 1/9$ independent of $n,m$. 
\end{conjecture}

\paragraph{}
For example, for $G=C_4 * C_5$, a graph with Euler characteristic $\chi(G)=0$
and Betti vector $(1,2,1)$ (it is homotopic to a 2-torus),
the Hodge determinants are
$(w_0,w_1,w_2,w_3)= ({\rm Det}(L_0),{\rm Det}(L_1),{\rm Det}(L_2),{\rm Det}(L_3))$,= \\
$(176084927365834800$,  \\
$306822144499476699689198835309067298521743360000$, \\
$1915862536237718536143850697507558601523200$, \\
$1099511627776)$. This leads to $A(G) = w_0^0 w_2^2/(w_1^1 w_3^3) = 1/9$.
The Dirac determinants are
$(v_0,v_1,v_2)=({\rm Det}(D_0),{\rm Det}(D_1), {\rm Det}(D_2)) =  \\
(176084927365834800$, $1742466826033449739173966643200$, $1099511627776)$.
This leads again to $A(G) = v_0 v_2/v_1 = 1/9$. 

\paragraph{}
For the homotopy cylinders $G = C_n * K_m$ we measure

\begin{conjecture}
$A(C_n * K_m) = n^2/m$.
\end{conjecture}

\paragraph{}
But things get more complicated even if $H$ is contractible. If $W_m$ is the 
wheel graph with $m$ vertices, then we see  \\
$A(C_4 * W_5) = 16*5/21$, $A(C_4 * W_6) = 16*6/26$, $A(C_4 * W_7) = 16*7/31$\\
$A(C_5 * W_5) = 25*5/21$, $A(C_5 * W_6) = 25*6/26$, $A(C_5 * W_7) = 25*7/31$ \\
suggesting

\begin{conjecture}
$A(C_n * W_m) = n^2 m/(4m+1)$.
\end{conjecture}

If $L_n$ is the linear graph with $n$ vertices, we see

\begin{conjecture}
$A(C_n * L_n) = n^2 m/(3m-2)$. 
\end{conjecture}

\begin{comment}
n=5; Do[s = NormalizeGraph[StrongProduct[CycleGraph[n], WheelGraph[m]]];
Print[{n, m}, AnalyticTorsion[s]/n^2], {m, 5, 7}]

n=5; Do[s = NormalizeGraph[StrongProduct[CycleGraph[n], LinearGraph[m]]];
Print[{n, m}, AnalyticTorsion[s]/(n^2*m)], {m, 2, 7}]
\end{comment}

\paragraph{}
{\bf Cayley's formula} tells that the number ${\rm Det}(L_0)$ of rooted $1$-dimensional 
trees in a complete graph satisfies ${\rm Det}(L_0)=n^{n-1}$. More generally, we have

\begin{lemma}[Generalized Cayley tree formula]
For the Hodge blocks: ${\rm det}(L_k(K_n))=n^{B(n,k+1)}$ for $k>0$ and $n^{n-1}$ for $k=0$. \\
For the Dirac blocks: ${\rm det}(D_k(K_n))=n^{B(n-1,k+1)}$ for all $k$. 
\end{lemma}

\begin{comment}
Import["~/text3/graphgeometry/all.m"];
F[n_]:=Module[{s,f}, s = NormalizeGraph[CompleteGraph[n]]; Log[n,MiniDet[s]]];
F1[n_]:=Table[Binomial[n-1,k+1],{k,0,n-2}] 
F[8]
F1[8]
G[n_]:=Module[{s,f}, s = NormalizeGraph[CompleteGraph[n]]; Log[n,HodgeDet[s]]];
G1[n_]:=Table[Binomial[n,k+1],{k,0,n-1}] 
G[8]
G1[8]
\end{comment}

This immediately implies: 

\begin{coro}
$A(K_n)=n$
\end{coro}
\begin{proof}
We can see this for the Hodge blocks
$$ A(K_n) = \prod_{k=0}^{n} {\rm Det}(L_k(K_n))^{k (-1)^k} $$
The identity could be seen by 
differentiating $\sum_{k=1}^n  B(n,k) x^k = (1+x)^n-1$ and setting $x=-1$.
Easier is to see it for the Dirac blocks
$$ A(K_n) = \prod_{k=0}^{n} {\rm Det}(D_k(K_n))^{(-1)^k} $$
which is $n^{\sum_{k=1}^n B(n-1,k) (-1)^k} = n^1$
%  G[n_]:=Sum[k Binomial[n,k] (-1)^k,{k,n}]  
\end{proof} 

\paragraph{}
In order to prove that $A(G)=|V(G)|$ is a valuation in the contractible case. 
we want to interpret $A(G)/|V(G)|$ is a geometric quantity which 
satisfies the {\bf counting property}:
$A(G \cup H) = A(G) + A(H) - A(G \cap H)$. 

\paragraph{}
Every finite simple graph $G$ defines its Whitney complex $W(G)$ and every
finite abstract simplicial complex $W$ defines the incidence graph $G(W)$
and $W \circ G$ or $G \circ W$ are
Barycentric refinements on the category of graphs or complexes.
We insist to remain in a combinatorial setting and chose
the language of graphs as this is much more intuitive.
This also follows early topologists like Whitney, Alexandroff or Hopf,
we like to think in terms of graphs (Gitterger\"uste) rather than
finite set of sets or geometric realizations. The later leaves
combinatorics and requires astronger axiom system ZFC.
When doing finite combinatorics, we do not need the infinity axiom.
It is also pedagogically simpler as what we do here is accessible
to anybody who has seen matrices, eigenvalues and determinants in
linear algebra. Simplicial complexes, CW complexes require more mathematical
maturity. 

\paragraph{}
Our geometric point of view is to see a graph $G$ as a geometric model
of a continuum like a compact Riemannian manifold $M$. 
The graph $G$ naturally comes with a geodesic metric and 
recovers differential geometric notions like tensors. Fundamentally, if we
look at space, we only can observe a finite set $V$ of points as well as relations
between these points given by an equivalence relation that if two points are
indistinguishable with a given accuracy. This define
$G$ and the metric. If space $M$ is a compact Riemannian manifold and 
the number of points is finite but $h$-dense with respect to some fixed
non-standard small $h>0$ in an axiomatic framework like Nelson's ZFC+IST or
ZF+SPOT, we can recover the Riemannian metric from the geodesic graph metric.

\paragraph{}
Not having any additional structure at first produces 
combinatorial problems which are not obscured by a particular choice of metric. 
The hope of course is always that some functional has interesting 
maxima or minima which somehow relate to physics. Natural functionals are 
Euler characteristic, Wu characteristic, average simplex 
cardinality \cite{AverageSimplexCardinality}, 
characteristic length \cite{KnillFunctional}. Related to the pseudo determinant
of the Dirac operator is torsion which is a super pseudo determinant of the
Dirac operator. 

\paragraph{}
One of our initial motivations was the question of 
relating the spectrum of the Hodge Laplacian $L$ of a graph $G$
with the combinatorial data. A good starting point is the case of spheres,
where we have Dehn-Sommerville relations involving the f-vector. 
Because of McKean-Singer, a natural quantity is the {\bf super determinant}
of Hodge operators. What is the geometric meaning if we weight the $k$'th term. 

\paragraph{}
When looking at the definition of analytic torsion, there is the strange
power of $k$ in the $k$'th term. This disappears if one moves from 
the {\bf Hodge blocks} $L_k = F_k^* F_k$ to the {\bf Dirac blocks} 
$D_k= d_k^* d_k$. Analytic torsion becomes so a very natural quantity
as it is just the {\bf super determinant of the Dirac operator} $D$ or the 
{\bf square of the super determinant of chain complex} defined by $d$. 

\paragraph{}
We made our first experiments with ${\rm Det}(D)$ when writing 
\cite{CauchyBinetKnill} and experimented with analytic torsion in the summer of 2015 and
saw then relations like $A(G)=|V|$ for contractible graphs or $A(C_n)=n^2$. Being
unable to prove the contractible case beyond the complete graph $K_n$, 
we moved on to other projects and only returned to it in December 2021. We realized
the proof of the 2-sphere case on December 25, 2021.
We learned about the linear algebra relating the $L_k$ 
with $D_k$ from \cite{Bunke2015} only on January 4th, 2022. 
We found then also the appendix of \cite{GelfandKapranovZwlevinsky},
which introduces the {\bf determinant of a chain complex}. 

\paragraph{}
The linear algebra switching from {\bf Hodge blocks} $L_k$ to {\bf Dirac blocks} 
$D_k$ is extremely important as it makes it clear {\bf why}
analytic torsion is such a natural quantity. We defined it as
the square of the determinant of a chain complex. Just because we like to work with 
rational numbers and not square roots, we continued to work with the 
squared analytic torsion. When doing experiments, we see for example
for the octahedron graph the torsion $A(G) = 3/4$. If we would do the experiments with 
the square root, we would see $\sqrt{A(G)} = \sqrt{3}/2$ and the connection 
to the $f$-vector $f_G = (6,12,8)$ would have been obscures. Taking the rational
numbers instead of the square roots looks like a small matter, but it was
essential when investigating the matter experimentally. 

\paragraph{}
As the above remarks have indicated, there are lots of open questions. 
We see that for discrete manifolds that 
analytic torsion also depends on topology. A good start for further
investigation is to see how torsion for $2$-manifolds depends on the 
structure of the manifold.
For homotopy tori $G$ obtained by taking the Shannon product $G=C_n * C_m$ of $C_n$ with $C_m$,
we always get $A(G) = 1/9$. The graph $G$ is three dimensional in nature with the same number 
of $n*m$ of vertices and tetrahedra $K_4$. 

\paragraph{}
Already Ray and Singer already suggested to study the analytic 
torsion for other differential complexes and not only the 
{\bf Euler complex}. We can look at it for 
{\bf Wu characteristic} $\omega(G) = \sum_{x \sim y} \omega(x)$, where
the sum is over all pairs $(x,y)$ of complete sub-graphs of $G$
which have a non-empty intersection and where $\omega(x) = (-1)^{\rm dim(x)}$.
Unlike Euler characteristic $\chi(G) = \sum_x \omega(x)$ which is a homotopy 
invariant, the Wu characteristic is not. For discrete manifolds with 
$(d-1)$-manifold boundary, it satisfies $\omega(G) = \chi(G)-\chi(\delta G)$. 

\paragraph{}
The second order analogue of the $f$-vector $f=(f_0,f_1, \dots f_d)$
is the $f$-matrix $f_{ij}$ matters which counts the number of 
intersections of $i$ simplices with $j$ simplices.  The second order
{\bf analytic torsion} is 
$$ A_2(G) = \prod_k {\rm Det}(L_k)^{k (-1)^{k+1}}  \; , $$
where $L_k$ are the blocks of the Laplacian $L=(d+d^*)^2$. 

We see that
for even dimensional spheres $A_2(G) = f_{00}/f_{dd}$ 
and for odd-dimensional spheres $A_2(G) = 1/(f_{00} f_{dd})$. 
We also see that for contractible graphs the situation is 
more subtle. We see that $A_2(K_n)=n/2^{n-1}$. 

\paragraph{}
Historically torsion was first considered for 3-manifolds by Reidemeister and
was then extended by Franz to higher dimensions. It originally involved 
a representation of a group acting on the manifold. 
Torsion usually is defined for manifolds with trivial cohomology
but equipped with a metric or with a unitary representation of the fundamental group. 
When considered for Riemannian manifolds, it involves also the volume.

\paragraph{}
The pioneering papers are not so easy to read. Reidemeister torsion and analytic torsion 
were identified in \cite{Mueller1978,Cheeger1979}. What we do here is much more elementary.
Analytic torsion for graphs and more generally for any finite differential complex
given by a finite sequence of derivative matrices $d_k$ only involves familiar linear algebra 
and is defined for arbitrary graphs or finite abstract simplicial complexes without 
additional structure. Especially, we never
actually need even to involve the continuum. The pseudo determinant is a product of eigenvalues
but it is also an entry in the characteristic polynomial defined by an integer matrix and so
computable as an integer without detour over eigenvalues. Analytic torsion is a rational number
explicitly computable in polynomial time from the simplicial complex (finding the Whitney complex
can be costly as finding cliques in a graph in general is NP complete). 

\paragraph{}
The pioneering paper \cite{Franz1935} which considers a cover
of a topological complex for which all Betti numbers $b_1,\dots,b_{d-1}$
with respect to some field $K$ are zero but where one still can have 
a non-trivial fundamental group. Examples are lense spaces.  
Franz then looks at basis changes for which the determinant
is in a fixed multiplicative subgroup of the field $K$.
Also \cite{DeRham1967} which is one of the later accounts of De Rham on 
torsion defines it for a group of units acting as automorphisms on a 
cellular complex. 

\paragraph{}
Milnor was one of the first, who picked up torsion, where de Rham
left off. In \cite{Milnor1961}, torsion was used to construct two manifolds
with boundary which are not diffeomorphic, even so the interiors are
diffeomorphic. Following Reidemeister, Franz and mostly 
de Rham one can define the torsion of a CW-complex $K$ equipped with 
a discrete group action $\Pi=\pi_1(K)$ so that $K/\Pi$ has only finitely 
many cells. Given also a multiplicative homomorphism from $\Pi$ to 
a commutative ring $P$ so that all the equivariant homologies 
$H_i(P \oplus_\Pi C_*(K))$ are zero. Then torsion is defined as 
a unit in $P$.

\paragraph{}
\cite{Milnor1966} summarizes the beginnings: 
{\it in 1935, Reidemeister \cite{Reidemeister1935}, Franz \cite{Franz1935} 
and de Rham \cite{DeRham1939} introduced 
the concept of "torsion" for certain finite simplicial complexes $X$.
(...) it is a kind of determinant which describes
the way in which the simplexes of $X$ are fitted together 
with respect to the action of the fundamental group.
(...) In 1950, J. H. C. Whitehead defined the "torsion" of a 
homotopy equivalence between finite complexes. This is a 
direct generalization of the Reidemeister \cite{Reidemeister1935}, 
Franz, and de Rham concept; but is a more delicate invariant. }

\paragraph{}
For compact oriented Riemannian manifolds $M$, the Ray-Singer torsion $T(M)$ of $M$ is
defined using the {\bf Ray-Singer determinant} 
\cite{RaySinger1971} which was introduced in 1971. They first refer to the Reidemeister-Franz
torsion as a function of certain representations of the fundamental group and then
introduce analytic torsion $T(M)$ then state $T(M_1 \times M_2) = T(M_1)^{\chi(M_2)}$
if $M_2$ is simply connected. The definition of $T(M)$ uses zeta regularized
determinants and only considered for odd dimensional manifolds as it is zero in 
the even dimensional case. In their definition of analytic torsion, \cite{RaySinger1971}
take a representation $O$ of the fundamental group $\pi_1$ by orthogonal matrices and
differential forms with values in the associated vector bundle. 
They assume that the Laplacian $\Delta$ has no zero eigenvalue so that the 
zeta function is analytic at $0$ allowing the definition. For modern approaches, see 
\cite{BurghelaFriedlanderKappelerMcDonald, Bunke2015}. 

\paragraph{}
The {\bf zeta function}
$\zeta_k(s) = \sum_{\lambda_k \neq 0} \lambda_k^{-s}$ 
of the $k$'th Laplacian of a manifold can be written as
$$ \zeta_k(s) = \Gamma(s)^{-1} \int_0^{\infty}
 t^{s-1} tr(e^{-t L_k}) \; $$ 
because 
$\int_0^{\infty} t^{s-1} e^{-t \lambda}) 
     = \Gamma(s) \lambda^{-s}$. The zeta function is
analytic except for some poles. Then ${\rm Det}(L_k)$
is defined as $e^{-\zeta_k'(0)}$ and analytic torsion 
as before. As pointed out in \cite{Witten2021}, the
individual determinants ${\rm Det}(L_k)$ are not always
defined as $0$ can be pole, but magically, the various
poles cancel. It goes without saying that computing the
torsion for a given manifold using the definitions is
almost impossible as we can not compute the eigenvalues
explicitly. For a general manifold, one has to be in a situation, where the 
heat kernel asymptotic are known.  
% x^s Integrate[t^(s-1) Exp[-t x],{t,0,Infinity}]

\paragraph{}
We can also look at the zeta function in the discrete case.
The Hodge block zeta function $\zeta_k(s) = \sum_{\lambda_j \neq 0} \lambda_j^{-s}$
of the hodge $L_k$ defines a Hodge zeta function
$\sum_{k} (-1)^k \zeta_k(s)$ which is not interesting as
by McKean-Singer, this is always constant zero. However, we can define
the {\bf super Hodge zeta function} of a graph as
$$ \zeta(s) = \sum_{k} (-1)^k \zeta_{L_k}(k s) \; . $$
which now the property that
$$  A = e^{-\zeta'(0)} \; . $$

\paragraph{}
Much more natural is the {\bf Dirac zeta function}
$$  \zeta(s) = \sum_{k} (-1)^k \zeta_{D_k}(s) \;   $$
which again satisfies 
$$  A = e^{-\zeta'(0)} \; . $$

\begin{figure}[!htpb]
\scalebox{0.7}{\includegraphics{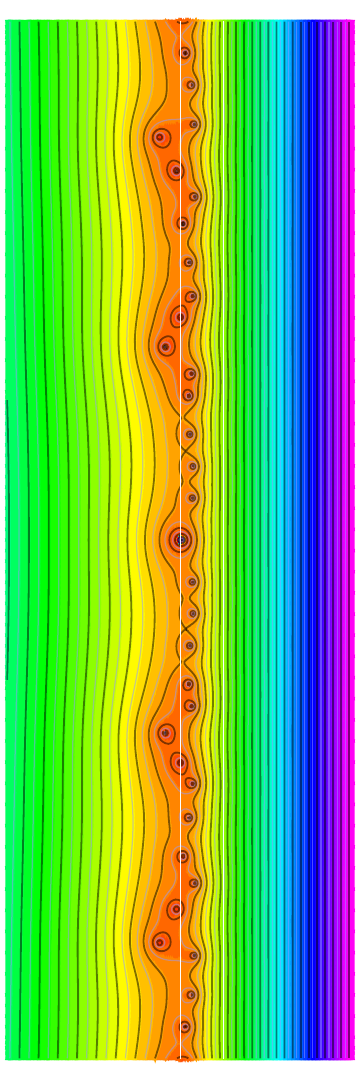}}
\scalebox{0.7}{\includegraphics{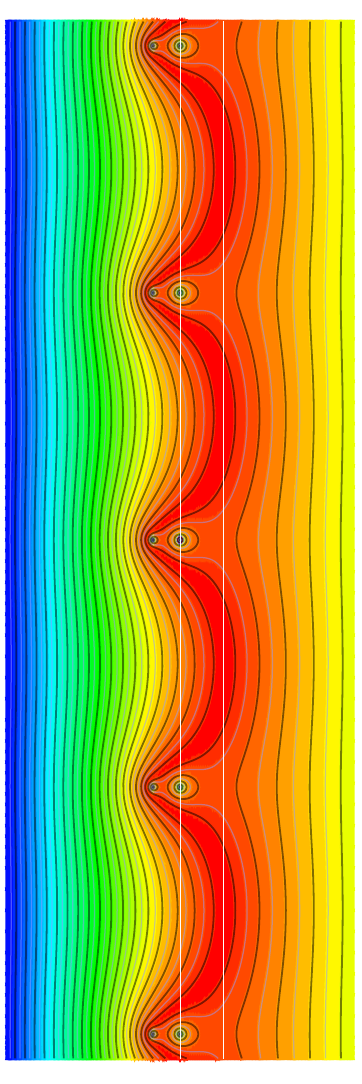}}
\scalebox{0.7}{\includegraphics{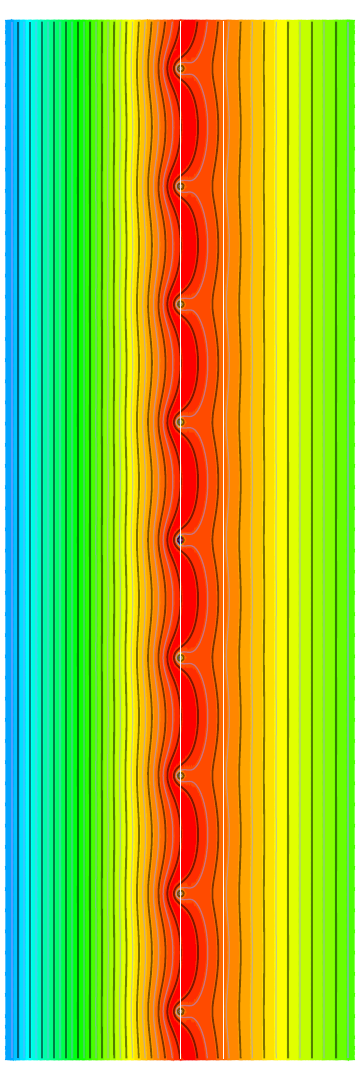}}
\label{Figure 4}
\caption{
Some Hodge zeta functions. First for the circular graph $C_{32}$
then for the triangle $K_3$ and then for the star graph $S_{10}$. 
}
\end{figure}

\paragraph{}
If $G=(V,E)$ is a finite simple graph, let $X_k$ denote the set of 
complete sub-graphs $K_{k+1}$ in $G$. 
The union $X=\bigcup_{k=0}^r X_k$ is a finite set of sets which is closed
under the operation of taking non-empty subsets. It is a {\bf finite abstract
simplicial complex}. If $f_k$ is the cardinality of $X_k$ then $f=(f_0, \dots, f_r)$
is the $f$-vector of $G$. The integer $r$ is the {\bf maximal dimension} and 
$r+1$ is the {\bf clique number}. The integer $n=\sum_{k=0}^r f_k$ is the
number of simplices in $X$. We fix a basis by assuming each $x \in X$ to be ordered.
The Dirac operator $D=d+d^*$ is a $n \times n$ matrix. It depends on the given order
but changing the order imposed on a simplex just produces 
an orthogonal change of basis.

\paragraph{}
Even so we work with real matrices we write $A^*$ for the transpose of a 
matrix $A$. The Dirac operator $D=(d+d^*)$ always is singular because the Hodge
Laplacian $L=D^2$ is. If $d_k$ denotes the exterior derivative from $k$-forms
to $(k+1)$-forms, then $d_k$ is a $f_{k+1} \times f_k$ matrix. 
The {\bf Dirac block} $D_k = d_k^* d_k$ is a $f_k \times f_k$ matrix and
is essentially isospectral to 
$D_k' = d_k d_k^*$ which is a $f_{k+1} \times f_{k+1}$ matrix.
We can extend $d_k$ and $d_{k-1}^*$ to matrices so that they are $n \times f_k$ matrices.

\paragraph{}
This produces the {\bf Dirac columns} $F_k = d_k + d_{k-1}^*$, which is a 
$n \times f_k$ matrix. Now, $F_k^* F_k$ is a $f_k \times f_k$ matrix
$(d_k^* + d_{k-1})(d_k +d_{k-1}^*) = d_k^* d_k + d_{k-1} d_{k-1}^*=L_k$. 
The matrix $F_k F_k^*$ is an $n \times n$ matrix which is 
essentially isospectral to $L_k$. It is a block diagonal matrix 
which is zero everywhere except for the blocks 
$D_k$ and $D_{k+1}'$. We see 
${\rm Det}(L_k) = {\rm Det}(D_k) {\rm Det}(D_{k-1}')$. 

\paragraph{}
Define {\bf pseudo super determinant} of $D$ as
$$ A(G) = {\rm SDet}(D) = \prod_k {\rm Det}(D_k)^{(-1)^k}  \; . $$
Compare with the {\bf super determinant}
$$ {\rm Det}(D) = \prod_k {\rm Det}(D_k) \; . $$

\paragraph{}
We learned about the following key connection first in 
\cite{Bunke2015} and then \cite{GelfandKapranovZwlevinsky}.

\begin{lemma}[Key lemma]
$A(G)= {\rm SDet}(D)$
\end{lemma}
\begin{proof}
\begin{eqnarray*}
A(G) &=& \prod_k {\rm Det}(L_k)^{k (-1)^{k+1}}  \\
     &=& \prod_k {\rm Det}(D_k)^{k (-1)^{k+1}} {\rm Det}(D_{k-1})^{k (-1)^{k+1}} \; . 
\end{eqnarray*}
This product telescopes $({\rm Det}(D_1) {\rm Det}(L_3) {\rm Det}(L_5) ...)$
appears in the nominator and 
$({\rm Det}(D_2) {\rm Det}(L_4) {\rm Det}(L_6) ...)$
\end{proof} 

\begin{coro}
b) ${\rm Det}(D)*A(G) = \prod_{k \; {\rm even}} {\rm Det}(D_k)^2$. \\
c) ${\rm Det}(D)/A(G) = \prod_{k \; {\rm odd}}  {\rm Det}(D_k)^2$. \\
\end{coro}

As references, look at \cite{GelfandKapranovZwlevinsky} (Appendix A) 
or \cite{BurghelaFriedlanderKappelerMcDonald,Bunke2015}. 

\begin{figure}[!htpb]
\scalebox{0.5}{\includegraphics{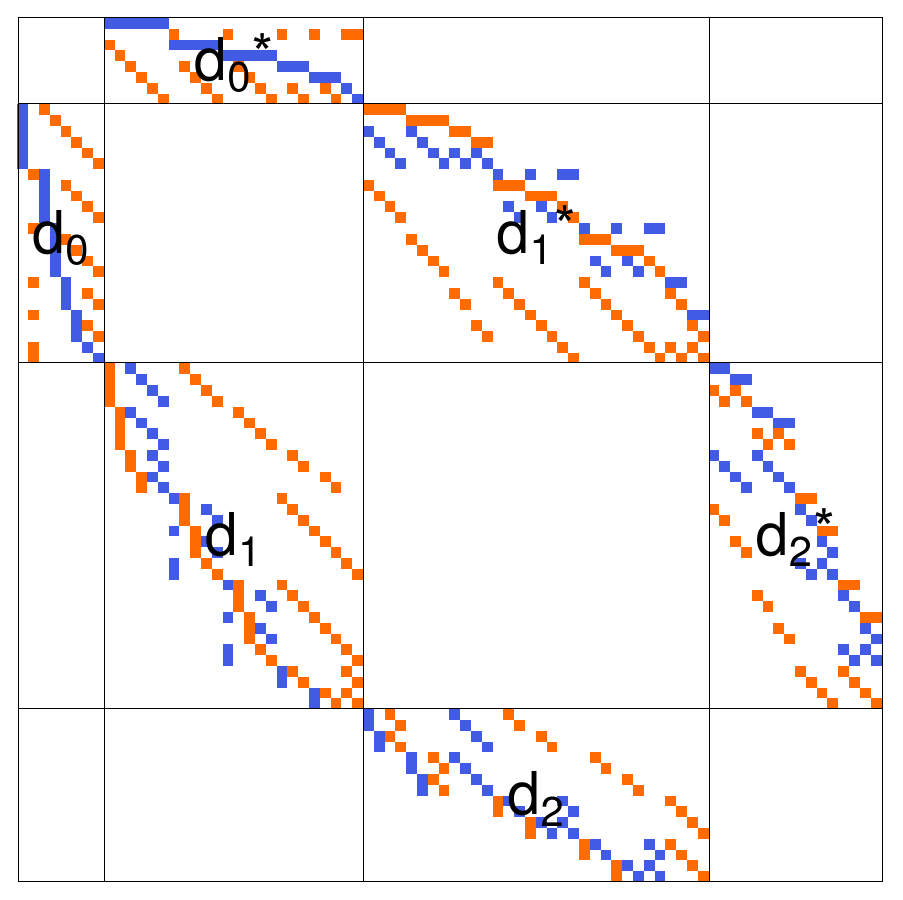}}
\scalebox{0.5}{\includegraphics{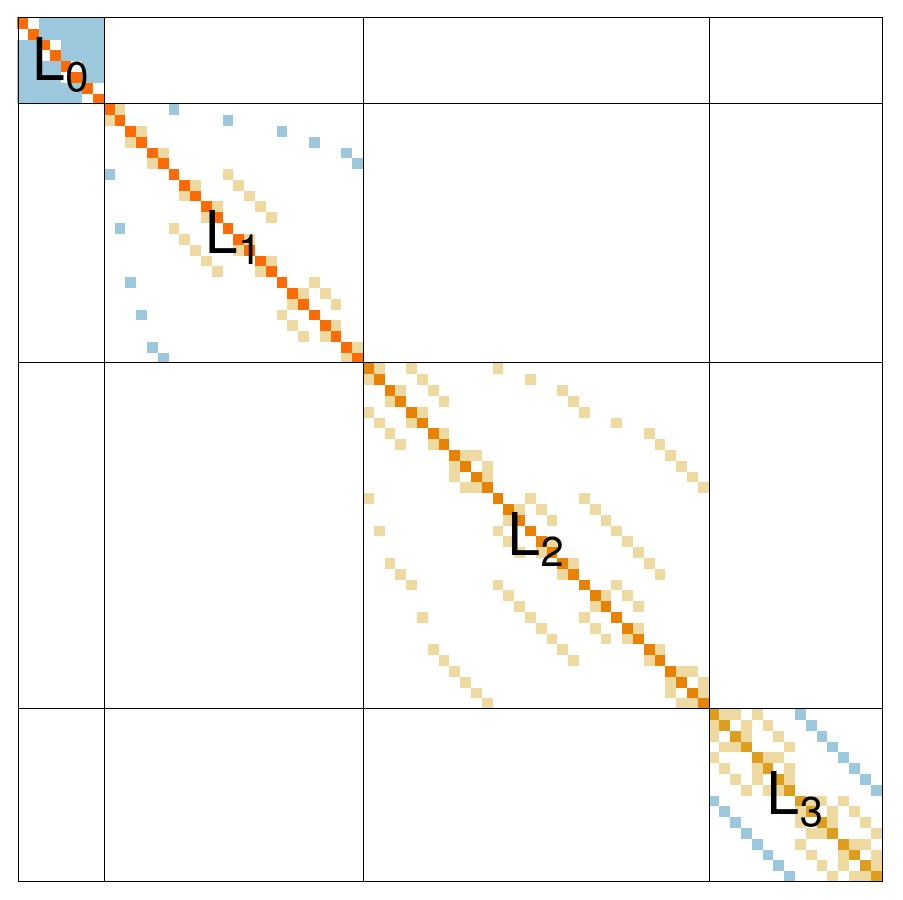}}
\label{Figure 4}
\caption{
The Dirac operator $D$ and the Hodge operator $L$ of a $3$-sphere $G$ with 
$f$-vector $f=(8,24,32,16)$ is a $80 \times 80$ matrix. 
The Hodge determinant vector is ${\rm Det}(L_k), k=0,1,2,3$ is \\
$(663552$, $2337302235907620864$, $2393397489569403764736$, $679477248)$. 
The Dirac determinant vector is ${\rm Det}(D_k), k=0,1,2$ is $(663552$, $3522410053632$, 
$679477248)$.  The analytic torsion is $A(G) = {\rm Det}(D_0) {\rm Det}(D_2)/{\rm Det}(D_1) =128$. 
This is the same than ${\rm Det}(L_1)^1 {\rm Det}(L_3)^3)/({\rm Det}(L_0)^0 {\rm Det}(L_2)^2)$. 
}
\end{figure}

\begin{figure}[!htpb]
\scalebox{1.0}{\includegraphics{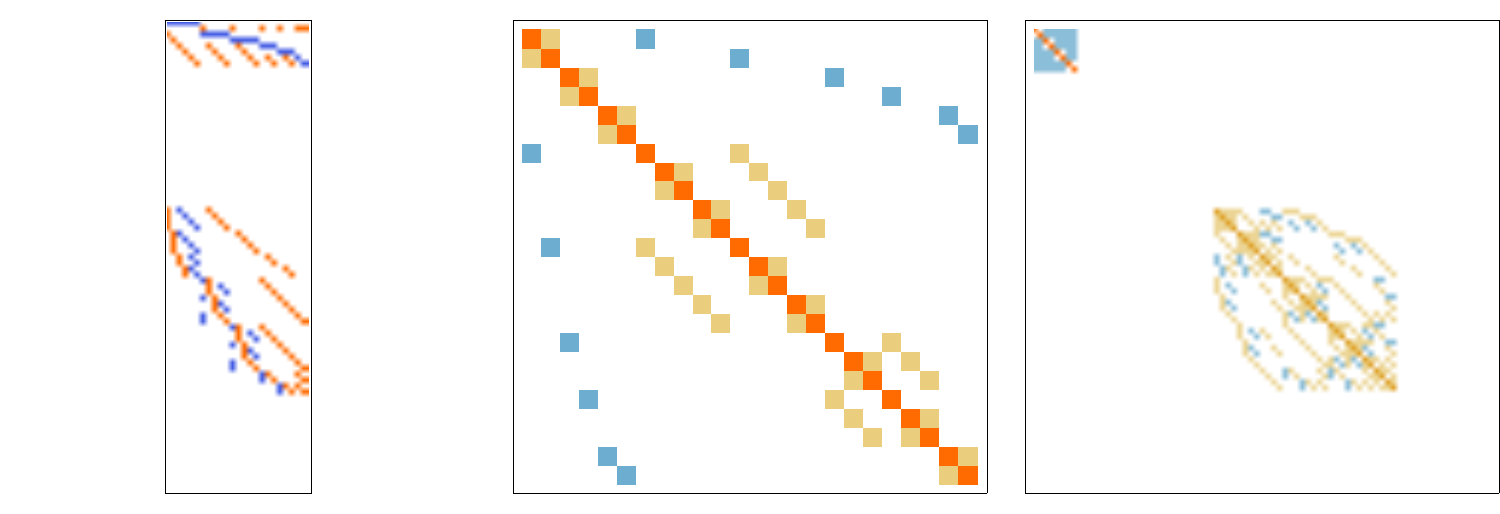}}
\label{Figure 4}
\caption{
The Dirac column $F_1$ belonging to $1$-simplices. The
next graphics shows the Hodge block $L_1 = F_1^* F_1$
which is a $f_1 \times f_1$ matrix. 
The last picture shows the $n \times n$ matrix $F_1 F_1^*$
which is a block diagonal matrix isospectral to $L_1=F_1^* F_1$ and
which contains the two Dirac blocks $D_1=d_1^* d_1$ and $D_0'=d_0 d_0^*$
This immediately shows that the pseudo determinants satisfies
${\rm Det}(L_1) = {\rm Det}(D_1) {\rm Det}(D_0)$. 
}
\end{figure}

\paragraph{}
If $D$ is the Dirac operator of a graph $G=(V,E)$. It is an $n \times n$ matrix. 
Let $A=D^{(1,1)}$ denote the $(n-1) \times (n-1)$-matrix in which the first row
and first column are deleted. The following lemma explains the factor $|V|$ 
appearing in the torsion of contractible graphs or spheres. It is a direct consequence
of the classical {\bf matrix tree theorem}.

\begin{lemma}[Shaving Dirac A] 
Let $A$ be defined from $D$ as above, then \\
a) ${\rm Det}(A) = {\rm Det}(D)/|V|$ \\
b) ${\rm SDet}(A) = {\rm SDet}(D)/|V|$
\end{lemma} 
\begin{proof}
Since the first {\bf Dirac block} $D_0 = d_0^* d_0$ agrees with the 
Kirchhoff matrix $L_0= d_0^* d_0$ and ${\rm Det}(L_0) = |V| {\rm Det}(L_0^{(1,1)})$ 
where $L_0^{(1,1)}$ is the $(f_0 -1) \times (f_0-1)$ matrix in which the first row
and column are deleted. For all $k=1, \cdots, |V|=f_0$, the determinant 
${\rm det}(L_0^{(k,k)})$ is the number of spanning trees rooted at the vertex $k$. 
The number ${\rm Det}(L_0)/|V|$ is the number of spanning trees in the graph. \\
To get a), note that ${\rm Det}(D) = \pm {\rm Det}(D_0) {\rm Det}(D_1) \cdots {\rm Det}(D_r)$ and 
that the matrices $D_1, \dots D_r$ are not affected by shaving off the first row and column. 
To get b), we use that ${\rm SDet}(D) = {\rm SDet}(D^{Even})/{\rm SDet}(D^{odd})$. 
\end{proof}

\paragraph{}
The dual story is when shaving away the last row and column of $D$. 
Let $B$ the Dirac operator in which the last row and last column are deleted. 
The next lemma explains the factor $|f_r| = |V'|$ appearing in torsion of spheres. 
Also this can be seen as a consequence of the matrix tree theorem. 

\begin{lemma}[Shaving Dirac B] 
If $G=(V,E)$ is a d-sphere,  and let $B$ be defined as above from $D$, then \\
a) ${\rm Det}(B)  = {\rm Det}(D)/|V'|$ \\
b) ${\rm SDet}(B) = {\rm SDet}(D)|V'|$  if $d$ is even.
c) ${\rm SDet}(B) = {\rm SDet}(D)/|V'|$ if $d$ is odd.
\end{lemma}

\paragraph{}
This is in general false. We need the last Betti vector to be $1$.
But it holds for torus graphs. 

\paragraph{}
So, in order to prove the result, we replace $D$ with $A$ in the 
contractible case and replace $D$ with $C$, the matrix in which the 
entire boundary has been shaved away. 
Let us introduce a new functional for contractible graphs 
$$ \phi(G) = {\rm SDet}(A) \; . $$
and for  spheres: 
$$ \phi(G) = {\rm SDet}(C) \; . $$
Now everything boils down to 

\begin{lemma}
For spheres or contractible spaces, we have $\phi(G) = 1$. 
\end{lemma}
\begin{proof}
For spheres and contractible spaces
we have an interpretation as trees and so a Meyer-Vietoris valuation
formula: $\phi(X \cup Y) = \phi(X) + \phi(Y) - \phi(X \cap Y)$. 
We can show this by induction. We can build up contractible graphs from
smaller contractible graphs. We can also build d-spheres by gluing two 
d-balls $X,Y$ (which are contractible) in such a way that 
$X \cap Y$ is a $(d-1)$-sphere.
The reason for the formula is that $\phi(G)=1$ now tells that there is 
a balance between {\bf even trees} and {\bf odd trees}. 
\end{proof}

\paragraph{}
We plan to follow up on the symmetry between even and odd trees in a
future work. There is more to say about the duality of higher dimensional 
spanning trees in spheres. 

\bibliographystyle{plain}

\end{document}